\documentclass{amsart}

\title[Determinant map on C*-algebras]{The kernel of the determinant map on certain simple C*-algebras}

\author{P. W. Ng}

\address{Mathematics Department\\
University of Louisiana at Lafayette\\
217 Maxim D. Doucet Hall\\
P. O. Box 41010\\
Lafayette, Louisiana\\
70504-1010\\
USA}

\email{png@louisiana.edu}

\newtheorem{thm}{Theorem}[section]
\newtheorem{prop}[thm]{Proposition}
\newtheorem{lem}[thm]{Lemma}
\newtheorem{df}{Definition}[section]
\newtheorem{rem}[thm]{Remark}
\newtheorem{cor}[thm]{Corollary}  
\newtheorem*{qu}{Question}

\newcommand{\A}{\mathcal{A}}
\newcommand{\B}{\mathcal{B}}
\newcommand{\C}{\mathcal{C}} 
\newcommand{\M}{\mathbb{M}}
\newcommand{\T}{\mathbb{T}}  
\newcommand{\F}{\mathcal{F}}
\newcommand{\G}{\mathcal{G}}
\newcommand{\I}{\mathcal{I}}
\newcommand{\J}{\mathcal{J}}
\newcommand{\INT}{\mathbb{INT}}
\newcommand{\Mul}{\mathcal{M}}

\newcommand{\tDS}{\widetilde{\Delta}}

\numberwithin{equation}{section}


\begin{document}

\begin{abstract}
Let $\A$ be a unital separable simple C*-algebra such that 
either
\begin{enumerate}
\item $\A$ has real rank zero, strict comparison and cancellation
of projections; or
\item $\A$ is TAI.  
\end{enumerate}

Let $\Delta_T : GL^0(\A) \rightarrow E_u/ T(K_0(\A))$ be the 
universal determinant of de la Harpe and Skandalis.

Then for all $x \in GL^0(\A)$,
$\Delta_T(x) = 0$ if and only if $x$ is the product of 8 multiplicative
commutators in $GL^0(\A)$.\\
 
We also have results for the unitary case and other cases. 
\end{abstract}

\maketitle

\section{Introduction}

Let $\A$ be a unital C*-algebra and let $x \in \A$. 

\begin{qu}
\makebox{  }
\begin{enumerate}
\item When is $x$ a finite sum of additive commutators? I.e.,
when is $x$ a sum of finitely many
elements of the form $ab - ba$ where $a,b \in \A$?
\item If $x$ is invertible (unitary) when is $x$ a finite product of
multiplicative commutators? I.e., when is $x$ a product of finitely
elements of the form $yzy^{-1} z^{-1}$ where $y, z \in \A$ are 
invertible (resp. unitary) elements?    
\end{enumerate}
\end{qu}

The first question has a long history, is connected to basic 
questions about the structure of C*-algebras, 
and is still a subject matter
of recent papers.  
(E.g., see \cite{Halmos}, \cite{BrownPearcy},
\cite{Halpern}, \cite{FackCommutators} 
\cite{FackDelaHarpe}, 
\cite{DykemaSkripka},  
\cite{MarcouxMurphy}, \cite{MarcouxSurvey}, \cite{CuntzPedersen} 
and the references therein.)  

   In this paper, we focus on the second question.
The first result in this direction is due to 
Brown and Pearcy who proved that every unitary operator on 
a separable infinite dimensional Hilbert space is a multiplicative
commutator of unitaries, i.e., has the form $v w v^* w^*$ where $v, w$ are 
unitary operators on the Hilbert space (\cite{BrownPearcyII}).
This was generalized by M. Broise who proved
that a von Neumann factor $\Mul$ is not of finite type I
if and only if every unitary operator in $\Mul$
is a finite product of multiplicative commutators of unitaries 
(\cite{Broise}).

   In \cite{FackDelaHarpe}, Fack and de la Harpe proved that 
if $\Mul$ is a type $II_1$ factor and $x \in \Mul$ is invertible,
then $x$ has Fuglede--Kadison determinant one if and only if
$x$ is a finite product of multiplicative commutators, i.e.,
a finite product of elements of the form $y z y^{-1} z^{-1}$
where $y, z \in \Mul$ are invertibles.    
(See \cite{FackDelaHarpe} Proposition 2.5.)

   de la Harpe, Skandalis and Thomsen generalized the above results
to classes of C*-algebras that are not necessarily von Neumann algebras.
For a unital C*-algebra $\A$, let $GL^0(\A), U^0(\A)$ denote the
connected component of the identity of the invertible group of $\A$
and the connected component of the identity of the unitary group of
$\A$ respectively; let $DGL^0(\A), DU^0(\A)$ denote the commutator
subgroups of $GL^0(\A)$ and $U^0(\A)$ respectively; and let   
$\Delta_T$ denote the universal determinant of $\A$, introduced
by de la Harpe and Skandalis in \cite{HarpeSkandalisI}.  (More
information about $\Delta_T$ and basic references 
can be found at the
end of this introduction.)
In \cite{HarpeSkandalisII} Theorem 6.6,
de la Harpe and Skandalis proved that if $\A$ is a unital simple
infinite dimensional AF-algebra and
$x \in GL^0(\A)$,  then
$\Delta_T(x) = 0$ if and only if 
$x$ is the product of four multiplicative commutators 
in $GL^0(\A)$. 
They have a similar result for $U^0(\A)$, when $\A$ is simple AF  
(\cite{HarpeSkandalisII} Proposition 6.7).    
Moreover, when $\A$ is simple AF,
$DGL^0(\A)$ and $DU^0(A)$ are both simple modulo their centres 
(\cite{HarpeSkandalisIII}). 
Finally, when $\A$ is a unital simple properly infinite C*-algebra,
both $GL^0(\A)$ and $U^0(\A)$ are perfect groups (\cite{HarpeSkandalisII}
Theorem 7.5 and Propositon 7.7).

In \cite{ThomsenCommutators}, Thomsen generalized 
de la Harpe and Skandalis' results to the class of 
unital C*-algebras $\A$ which have the following properties:
\begin{enumerate}
\item $\A$ is an inductive limit where the
building blocks have the form 
$\M_{n_1}(C(X_1)) \oplus \M_{n_2}(C(X_2)) \oplus ... \oplus 
\M_{n_k}(C(X_k))$ 
such that each $X_j$ is a compact metric space with covering
dimension $dim(X_j) \leq 2$ and $H^2(X_j, \mathbb{Z}) = 0$.
\item $K_0(\A)$ has large denominators.
\end{enumerate}
Henceforth, we will call the above class of C*-algebras ``Thomsen's class".

In \cite{ThomsenCommutators} Theorem 3.4, using fundamental results
in classification theory that Thomsen developed,
it was proven that for a C*-algebra
$\A$ in Thomsen's class, for $x \in GL^0(\A)$ (or $x \in U^0(\A)$),
$\Delta_T(x) = 0$ if and only if $x$ is a finite product of
commutators in $GL^0(\A)$ (respectively in $U^0(\A)$).
We note that these results (unlike the result of, say,
\cite{HarpeSkandalisII} Theorem 6.6 which gives four) 
does not give a bound on 
the number of commutators.  Moreover, the argument itself does
not give such a bound.  Finally, in \cite{ThomsenCommutators} Theorem 4.1
and Theorem 4.3, it was proven that for a C*-algebra $\A$ 
in Thomsen's class, $DGL^0(\A)$ and $DU^0(\A)$ are both simple modulo their
centres.
 
In this paper, we generalize the results of 
\cite{HarpeSkandalisII} and  \cite{ThomsenCommutators} to
the class of simple TAI-algebras and the class of simple unital
C*-algebras with real rank zero, strict comparison and cancellation
of projections.  (The definition of ``TAI-algebra" is in 
Definition \ref{df:TAI}.)  These are large classes of C*-algebras
which have been important in the classification program. (E.g.,
the C*-algebras in \cite{ElliottEvans} and \cite{ElliottGongLi}
belong to these classes.)  These classes also include
the classes in \cite{HarpeSkandalisII} and 
\cite{ThomsenCommutators} (in the simple finite case).  
Our main result is the following:
Let $\A$ be a unital simple separable C*-algebra such that
either (a) $\A$ is TAI or (b) $\A$ has real rank zero, strict
comparison and cancellation of projections.
Let $x \in GL^0(\A)$.
Then $\Delta_T(x) = 0$ if and only if 
$x$ is the product of eight multiplicative commutators in 
$GL^0(\A)$.  (See Theorem \ref{thm:MainTAITh} and Theorem
\ref{thm:MainRR0Th}.)
We note that unlike the (nonetheless interesting)  
results in \cite{ThomsenCommutators},
there is a bound (eight) on the number of commutators. 
It is an open question whether we can reduce the bound. 
We also have results in the unitary case.  
(See also  Theorem \ref{thm:TAIFirstTh}
and Theorem \ref{thm:RR0FirstTh}.) 
 
The arguments in our paper extensively use techniques from
classification theory, including a difficult uniqueness theorem from
the literature  
(Theorem \ref{thm:LinAUE}). 

We end this section by giving some basic references and
fixing 
some notation and definitions which we will use throughout this paper.

A basic reference for the de la Harpe--Skandalis determinant
is \cite{HarpeSkandalisI}.  A good summary can also be found
in \cite{HarpeSurvey}.   A basic reference for TAI algebras
is \cite{LinTAI}.  

We now fix some notation and definitions.  We refer the reader to
the references given above for more details.  
For a unital C*-algebra, $\A$ and for 
$n \in \{ 1, 2, .... \} \cup \{ \infty \}$,
let $U_n(\A), U^0_n(\A), GL_n(\A), GL_n^0(\A)$        
be the unitary group, the connected component of the identity of the unitary
group, the group of invertibles, and connected component of the identity
of the group of invertibles respectively of $\M_n(\A)$. 
Oftentimes, we use $U(\A), U^0(\A), GL(\A), GL^0(\A)$ to abbreviate
$U_1(\A), U^0_1(\A), GL_1(\A), GL^0_1(\A)$ respectively. 
Also, for a group $G$ and for $x, y \in G$, we let
$(x,y)$ denote the multiplicative 
commutator $(x,y) =_{df} x y x^{-1} y^{-1}$.
We let $DG$ denote the commutator
subgroup of $G$, i.e., the subgroup of $G$ generated by the
multiplicative commutators $(x,y)$ where $x,y \in G$. 
(E.g., $D U^0(\A)$ is the commutator
subgroup of $U^0(\A)$.) 

For a Banach space $E$,  a tracial continuous linear function
$\tau : \A \rightarrow E$,
and for a piecewise continuously differentiable curve 
$\xi : [t_0,t_1] \rightarrow GL^0_{\infty}(\A)$, we let 
$\widetilde{\Delta}_{\tau}(\xi) =_{df}
\frac{1}{2\pi i} \int_{t_0}^{t_1} \tau(\xi'(t) \xi(t)^{-1}) dt \in E$
(section 1 of \cite{HarpeSkandalisI}; see also section 6
of \cite{HarpeSurvey}).
By \cite{HarpeSkandalisI} Lemma 1 (c) (also \cite{HarpeSurvey} Lemma 10(iii)),
$\widetilde{\Delta}_{\tau}(\xi)$ depends only on the homotopy class
of $\xi$ (with endpoints fixed).
This (and a form of Bott periodicity) then induces a group homomorphism
$\Delta_{\tau} : GL^0_{\infty} (\A) \rightarrow E / \tau(K_0(\A))$ 
(\cite{HarpeSkandalisI} Proposition 2; also
\cite{HarpeSurvey} Theorem 13).   

Let $E_u$ denote the Banach space quotient of $\A$ 
by the closed linear span of the additive commutators 
$[a,b] =_{df} ab - ba$, $a, b \in \A$, i.e.,
$E_u =_{df} \A / \overline{[\A, \A]}$.
Let $T : \A \rightarrow E_u$  denote the natural quotient map.
($T$ is called the \emph{universal tracial continuous linear map}.)
From the above, we have a group
homomorphism $\Delta_T : GL^0_{\infty} (\A) \rightarrow E_u/ T(K_0(\A))$
which is called the \emph{universal de la Harpe--Skandalis determinant}.
Throughout this paper, we will simply call 
$\Delta_T$ the \emph{de la Harpe--Skandalis} determinant.
(We note that this determinant has been useful in classification theory.
See, for example, \cite{LinAsymptotic}, \cite{LinNiu}.)

Next, for a unital C*-algebra, we let $T(\A)$ denote the 
simplex of tracial states on $\A$.   

We let $\T$ denote the unit circle of the complex plane; i.e., 
$\T =_{df} \{ z \in \mathbb{C} : |z| = 1 \}$.

Throughout this paper, we let $\INT$ denote the class of C*-algebras of
the form $\bigoplus_{j=1}^m \B_j$, where for each $j$, $\B_j \cong \M_{n_j}$
or $\B_j \cong \M_{n_j}(C[0,1])$ for some positive integer $n_j$.

The following notion is due to Lin:

\begin{df}  A unital simple C*-algebra $\A$ is said to be \emph{tracially AI}
(TAI) if for any $\epsilon > 0$, for any finite subset $\F \subset \A$,
and for any nonzero
positive element $a \in \A_+$, there exists a projection $p \in \A$
and a C*-subalgebra $\I \in \INT$ with $1_{\I} = p$
such that
\begin{enumerate}
\item $1 - p$ is Murray-von Neumann equivalent to a projection in $\overline{a
\A a}$.
\item $\| px - xp \| < \epsilon$ for all $x \in \F$, and
\item $pxp$ is within $\epsilon$ of an element of $\I$, for all $x \in \F$,
\end{enumerate}
\label{df:TAI}
\end{df}

(Note:  In the above definition, ``AI" abbreviates ``approximately interval".)

Every simple unital TAI-algebra is quasidiagonal,
has real rank at most one,
stable rank one, property (SP),
and strict comparison (of projections by tracial states).
The $K_0$ group of a simple unital TAI-algebra has weak unperforation
and the Riesz Interpolation Property.  Many simple C*-algebras are TAI;
in particular every simple unital AH-algebra with bounded dimension growth
is TAI. (E.g., the algebras in 
\cite{ElliottEvans} and \cite{ElliottGongLi} are TAI.)
For these and other basic results about TAI-algebras, we refer
the reader to \cite{LinTAI}.

\begin{rem}  By \cite{LinTAI} Corollary 3.3, for the C*-algebra
$\I$  in
Definition \ref{df:TAI},
the matrix sizes of the summands of $\I$ can be taken
to be arbitrarily large.
I.e., for every $L \geq 1$, we can find an $\I$ satisfying
the conditions in Definition \ref{df:TAI} such that
every irreducible representation of $\I$ has dimension
greater than $L$ (i.e., the image of any irreducible representation
of $\I$ has the form $\M_k$ with $k \geq L$).
\label{rem:TAIMatrixSize}
\end{rem}

A final note:  In the results that follow, we will often state the
result in general, but only prove it in the infinite dimensional case.

\section{The TAI Case} 








\begin{lem} There exist two continuous functions 
$v, w : (-\pi/2, \pi/2) \rightarrow SU(2)$ such that 
\[
(v(t), w(t)) = 
\left[  \begin{array}{cc}
e^{i t} & 0 \\
0 & e^{-it} 
\end{array}
\right] 
\]
$$\| v(t) - 1 \|, \makebox{  } \| w(t) - 1 \|  \leq | e^{it} - 1 |^{1/2}$$
and 
$$v(0) = w(0)=  1$$
for all $t \in (-\pi/2, \pi/2)$.
\label{lem:DHSDiagonalCommutator}
\end{lem}

\begin{proof}
This follows from \cite{HarpeSkandalisII} Lemma 5.13.  
(Note that that $v_j(0) = 1$ for $j = 1,2$ 
follows from the inequalities.) 
\end{proof}

\begin{cor}  Let $\alpha \in \T$. 

Then there exist unitaries $v, w \in \M_2 (\mathbb{C})$ such that  
\[
(v, w) = \left[ \begin{array}{cc} \alpha & 0 \\
0 & \overline{\alpha} \end{array} \right] 
\]  

If, in addition, $|\alpha - 1 | < \sqrt{2}$, then
we may choose $v, w$ so that 
$$\| v - 1 \|, \makebox{  } \| w - 1 \| \leq |\alpha - 1|^{1/2}$$
\label{cor:MatrixCommutator}
\end{cor}

\begin{proof}

   If $| \alpha - 1 | < \sqrt{2}$ (i.e., the principal argument of 
$\alpha$ is in $(-\pi/2, \pi/2)$), then 
the result follows from Lemma \ref{lem:DHSDiagonalCommutator}. 

   For general $\alpha \in \T$, we note that 
\[
\left[  \begin{array}{cc} \alpha & 0 \\
0 & 1 \end{array} \right]
\left[ \begin{array}{cc} 0 & 1 \\
1 & 0 \end{array} \right]
\left[  \begin{array}{cc} \overline{\alpha} & 0 \\
0 & 1 \end{array} \right]
\left[ \begin{array}{cc} 0 & 1 \\
1 & 0 \end{array} \right]
= \left[ \begin{array}{cc} \alpha & 0 \\ 0 & \overline{\alpha} 
\end{array} \right].
\] 
\end{proof}

We now fix a notation.  Let $X$ be a metric space and let 
$S \subseteq X$ be a subset.  For every $\delta > 0$, let 
$N(S, \delta)$ denote the \emph{$\delta$-neighbourhood} of 
$S$; i.e., 
$N(S, \delta) =_{df} \{ t \in X : dist(t, S) < \delta \}$. 

\begin{lem}
Let $\theta : [0,1] \rightarrow \mathbb{R}$ be a continuous map.

Then there exists $v_1, w_1, v_2, w_2, v_3, w_3, v_4, w_4
 \in U(\M_2 (C[0,1]))$
such that
\[
(v_1(s), w_1(s))(v_2(s), w_2(s))(v_3(s), w_3(s))(v_4(s), w_4(s)) = 
\left[  \begin{array}{cc} e^{i\theta(s)} & 0 \\
0 & e^{-i \theta(s)} 
\end{array}
\right] 
\]  
for all $s \in [0,1]$.   

  Moreover, if there exists an open set $G \subseteq [0,1]$ such that
$\theta(s) = 0$ for all $s \in [0,1] - G$, then 
for every $\delta > 0$, we can choose the unitaries so that 
$w_1 = w_3 = \left[ \begin{array}{cc} 0 & 1 \\ 1 & 0 \end{array} \right]$,
and $v_k(s) = w_j(s) = 1$ for all $s \in [0,1] - N(G, \delta)$
and for $1 \leq k,j \leq n$ with $j \neq 1,3$. 
\label{lem:GeneralDiagonalCommutator}
\end{lem}

\begin{proof}

    Let $O_1, O_2, ..., O_n$ be an open covering
of $[0,1]$ such that 
for each $j$ with $1 \leq j \leq n$,
there exists an angle $\theta_i$ so that 
$\theta(s) \in \theta_j + [-\pi/4, \pi/4]$ for all $s \in O_j$.
 
Since $[0,1]$ has covering dimension one, 
taking a refinement of the open cover if necessary, we may assume that 
each point in $[0,1]$ is contained in at most two of the $O_j$.
Moreover, rearranging the $O_j$
if necessary, we may assume that 
for $j, j'$ such that $|j - j'| \geq 2$,
$O_j \cap O_{j'} = \emptyset$.

Let $\{ f_j \}_{j=1}^n$ be a partition of unity for $[0,1]$
subordinate to $\{ O_j \}_{j=1}^n$.  

We have that for $1 \leq j \leq n$,
$\theta(s) - \theta_j \in [-\pi/4, \pi/4]$ for all $s \in O_j$.
Hence,
$f_j(s)(\theta(s) - \theta_j) \in [-\pi/4, \pi/4]$ for all $s \in [0,1]$.
Let $v, w : (-\pi/2, \pi/2) \rightarrow SU(2)$ be the
continuous functions from Lemma \ref{lem:DHSDiagonalCommutator}.
Let $\widetilde{v_j}(s) =_{df} v(f_j(s)(\theta(s) - \theta_j))$
and 
$\widetilde{w_j}(s) =_{df} w(f_j(s)(\theta(s) - \theta_j))$ 
for $s \in [0,1]$.   
Hence, by Lemma \ref{lem:DHSDiagonalCommutator},  
\[
(\widetilde{v_j}(s), \widetilde{w_j}(s)) 
=
\left[
\begin{array}{cc}
e^{i f_j(s)(\theta(s) - \theta_j)} & 0 \\
0 & e^{-i f_j(s)(\theta(s) - \theta_j)}
\end{array}
\right]
\]
for all $s \in [0,1]$.

Note that for $1 \leq j \leq n$,  
\[
\left[
\begin{array}{cc}
e^{i f_j(s)\theta_j} & 0 \\
0 & e^{-i f_j(s)\theta_j}
\end{array}
\right]
=
\left(  \left[ \begin{array}{cc} e^{i f_j(s)\theta_j} & 0 \\
0 & 1  
\end{array}
\right],
\left[
\begin{array}{cc}
0 & 1 \\
1 & 0
\end{array}
\right]
\right)
\]
for all $s \in [0,1]$.  

Also, note that $\prod_{j=1}^n e^{if_j(s) \theta_j} e^{i f_j(s) (\theta(s)-
\theta_j)} = e^{i \theta(s)}$ for all $s \in [0,1]$.  

Thus, we can take
$v_1 =_{df}
\prod_{j \makebox{  odd } } diag( e^{i f_j \theta_j}, 1)$,
$w_1 = w_3 =_{df} \left[ \begin{array}{cc} 0 & 1 \\ 1 & 0 \end{array} \right]$.
$v_2 =_{df} \prod_{j \makebox{  odd } } \widetilde{v_j}$,
$w_2 =_{df} \prod_{j \makebox{  odd } } \widetilde{w_j}$,
$v_3 =_{df} \prod_{j \makebox{  even } } diag( e^{i f_j \theta_j}, 
1)$,
$v_4 =_{df} \prod_{j \makebox{  even } } \widetilde{v_j}$,
and 
$w_4 =_{df} \prod_{j \makebox{  even } } \widetilde{w_j}$. 

    Suppose, in addition, that
$G \subseteq [0,1]$ is an open subset such that
$\theta(s) = 0$ for all $s \in [0,1] - G$.
 Let $\delta > 0$ be given.
Let $g : [0,1] \rightarrow [0, \infty)$ be a continuous function
with $0 \leq g \leq 1$ such that (i.) $g(s) = 1$ for all $s \in G$ and
(ii.) $g(s) = 0$ for all $s \in [0,1] - N(G, \delta/2)$. 
In the definitions of $v_j, \widetilde{v_j}, w_j, \widetilde{w_j}$
($1 \leq j \leq n$) above, replace every occurrence 
of $f_j$ with the (pointwise product) $g f_j$ ($1 \leq j \leq n$).
Then $v_k(s) = w_j(s) = 1$, for all $s \in [0,1] - N(G, \delta)$
and for $1 \leq k, j \leq n$ with $j \neq 1, 3$.
\end{proof}

\begin{lem} Let $\phi_k : [0,1] \rightarrow \mathbb{R}$ ($1 \leq k \leq m$)
be continuous maps such that 
$$\phi_1(s) \leq \phi_2(s) \leq \phi_3(s) \leq ... \leq \phi_m(s)$$
and
$$\sum_{k=1}^m \phi_k(s) = 0$$
for all $s \in [0,1]$. 

Then we have the following:
\begin{enumerate}
\item There exist $v_j, w_j \in U^0(\M_m(C[0,1]))$ ($1 \leq j \leq 16$)
such that 
$$diag(e^{i \phi_1}, e^{i \phi_2}, ... , e^{i \phi_m})  =
\prod_{j=1}^{16} (v_j, w_j).$$
(Here, $\prod_{j=1}^{16} (v_j, w_j) = (v_1,w_1)(v_2, w_2) ... 
(v_{16}, w_{16})$.)\\

\item
Suppose, in addition, that $ran(\phi_k) \subset (\pi/2, \pi/2)$
for $1 \leq k \leq m$.

Then there exist $v_1, w_1, v_2, w_2, v_3, w_3, v_4, w_4 \in U^0(\M_m(C[0,1]))$
such that 
$$diag(e^{i \phi_1}, e^{i \phi_2}, ... , e^{i \phi_m})  =
(v_1, w_1) (v_2, w_2) (v_3, w_3) (v_4, w_4)$$
and
$$\| v_j - 1_{\A} \|, \makebox{  } \| w_j - 1_{\A} \|
 \leq \sqrt{2} \| u - 1_{\A} \|^{1/2}$$

for $1 \leq j \leq 4$,
where 
$\A =_{df} \M_m (C[0,1])$ and 
$u =_{df} diag(e^{i\phi_1}. e^{i \phi_2}, ..., e^{i \phi_m})$.
\end{enumerate}
\label{lem:INTDiagonal}
\end{lem}

\begin{proof}

   The proof is a modification of the arguments of 
\cite{ThomsenCommutators} Lemma 2.7 and Lemma 2.8, where we additionally
use Lemma \ref{lem:DHSDiagonalCommutator} and
Lemma \ref{lem:GeneralDiagonalCommutator}. 
(Indeed, the proof of Part (2) is contained in the proof of
\cite{ThomsenCommutators} Lemma 2.8.)
We provide the argument for the convenience of the reader.

     If $diag(e^{i\phi_1}, e^{i\phi_2}, ..., e^{i\phi_m}) = 1_{\A}$ 
then we are done.  Hence, let us assume that
$diag(e^{i\phi_1}, e^{\phi_2}, ..., e^{i \phi_m}) \neq 1_{\A}$.

     Choose $\delta > 0$ small enough so that 
$\delta < \pi/4$, 
$|e^{i\delta} - 1| < \| u - 1 \|$ and
$| e^{-i\delta} - 1 | < \| u - 1 \|$.
Moreover, if $ran(\phi_j) \subset (-\pi/2, \pi/2)$ we require
$ran(\phi_j) \pm \delta \subset (-\pi/2, \pi/2)$.

    Firstly, note that for each $s \in [0,1]$, since
$\sum_{k=1}^m \phi_k(s) = 0$, there is a permutation $\sigma$ of 
$\{ 1, 2, ..., m \}$ ($\sigma$ is dependent on $s$) such that 
$$\phi_1(s) \leq \sum_{k=1}^l \phi_{\sigma(k)}(s) \leq \phi_m(s)$$
for $1 \leq l \leq m$.

   Since $[0,1]$ is compact, let $\{ O_j \}_{j=1}^n$ be an open covering
of $[0,1]$ and for $1 \leq j \leq n$, let $x_j \in U(\M_m)$ be a permutation
unitary and let $\sigma_j$ be a permutation of $\{ 1, 2, ..., m \}$
such that 
$$x_j diag(\phi_1, \phi_2, ..., \phi_m) x_j^* = 
diag(\phi_{\sigma_j(1)}, \phi_{\sigma_j(2)}, ..., \phi_{\sigma_j(m)})$$
and
$$\phi_1(s) - \delta \leq \sum_{k=1}^l \phi_{\sigma_j(k)}(s) \leq \phi_m(s)
+ \delta$$
for $1 \leq l \leq m$ and for all $s \in O_j$.  

Let $\gamma > 0$ be given. 
Since $[0,1]$ has covering dimension one, taking refinements,  
permuting and contracting the $O_j$s and contracting $\gamma > 0$
if necessary, we may assume that 
if $|j - j'| \geq 2$ then $N(O_j, \gamma) \cap N(O_{j'}, \gamma) = \emptyset$.
  
Let $\{ f_j \}_{j=1}^n$ be a partition of unity of $[0,1]$ subordinate
to $\{ O_j \}_{j=1}^n$.  
For $1 \leq j \leq n$, 
let $a_j \in \M_m(C[0,1])$ be the self-adjoint element given by
$$a_j =_{df} f_j diag( \phi_1, \phi_2, ..., \phi_m)$$
and for $1 \leq k \leq m$, let  
$$\psi_{j,k} =_{df} f_j \sum_{l=1}^k \phi_{\sigma_j(l)}.$$
Hence,
$$diag(\phi_1, \phi_2, ..., \phi_m) = \sum_{j=1}^n a_j$$
and
for $1 \leq j \leq n$,
\begin{equation}
x_j a_j x_j^* = diag(\psi_{j,1}, - \psi_{j,1},  \psi_{j,3}, -\psi_{j,3},
...) + diag(0, \psi_{j,2}, -\psi_{j,2}, \psi_{j,4}, -\psi_{j,4}, ...)
\label{equ:A-ADecomposition}
\end{equation} 
where the first diagonal ends with zero if $m$ is odd, and
the second diagonal ends with zero if $m$ is even.
 
We consider the two cases (or parts) in the statement of the Lemma.

\emph{Case 1 or Part (2):}  Suppose that $ran(\phi_k) \subset (-\pi/2, \pi/2)$
for $1 \leq k \leq m$.

  By (\ref{equ:A-ADecomposition}) and by Lemma \ref{lem:DHSDiagonalCommutator},
we have that for $1 \leq j \leq n$ and $1 \leq l \leq 2$,
there exist unitaries $v_{j,l}, w_{j,l} \in \M_m(C[0,1])$ such that 
$$x_j^*diag(e^{i\psi_{j,1}}, e^{-i\psi_{j,1}}, e^{i\psi_{j,3}}, 
e^{-i\psi_{j,3}}, ...)x_j 
= (v_{j,1}, w_{j,1})$$
$$x_j^* diag(1, e^{i\psi_{j,2}}, e^{-i\psi_{j,2}}, e^{i\psi_{j,4}}, 
e^{-i\psi_{j,4}}, ...)x_j  
= (v_{j,2}, w_{j,2})$$
and for $l = 1,2$,
$$v_{j,l}(s) = w_{j,l}(s) = 1$$
for all $s \in [0,1] - O_j$.
Moreover, by Lemma \ref{lem:DHSDiagonalCommutator} and 
by our choice of $\delta$,
for $1 \leq j \leq n$ and $l = 1,2$,
\begin{eqnarray*}
& & \| v_{j,l} - 1 \| \\
& \leq & max \{ |e^{i(\phi_1(s) - \delta)} - 1 |^{1/2},
|e^{i(\phi_m(s) + \delta)} - 1 |^{1/2} : s \in O_j \} \\
& \leq & max \{ |e^{i\phi_1(s)} (e^{-i \delta} - 1)| + |e^{i \phi_1(s)} - 
1|,  |e^{i\phi_m(s)} (e^{i \delta} - 1)| + |e^{i \phi_m(s)} - 1| : s 
\in O_j \}^{1/2}\\ 
& \leq & \sqrt{2} \| u - 1 \|^{1/2}.
\end{eqnarray*}
Similarly, $$\| w_{j,l} - 1 \| \leq \sqrt{2} \| u - 1 \|^{1/2}.$$

Now let $v_1 =_{df} \prod_{j \makebox{  odd  }} v_{j,1}$,
$w_1 =_{df} \prod_{j \makebox{  odd  }} w_{j,1}$, 
$v_2 =_{df} \prod_{j \makebox{  even  }} v_{j,1}$,
$w_2 =_{df} \prod_{j \makebox{  even  }} w_{j,1}$,
$v_3 =_{df} \prod_{j \makebox{  odd  }} v_{j,2}$,
$w_3 =_{df} \prod_{j \makebox{  odd  }} w_{j,2}$, 
$v_4 =_{df} \prod_{j \makebox{  even  }} v_{j,2}$,
$w_4 =_{df} \prod_{j \makebox{  even  }} w_{j,2}$. 

Then for $1 \leq l \leq 4$,
$\| v_l - 1 \| \leq \sqrt{2} \| u - 1 \|^{1/2}$ and
$\| w_l - 1 \| \leq \sqrt{2} \| u - 1 \|^{1/2}$.
Also, 
$$u = \prod_{l=1}^4 (v_l, w_l)$$
as required.\\

\emph{Case 2 or Part (1):}  General case.

The proof for this case is the same as that of Case 1, except that
we replace Lemma \ref{lem:DHSDiagonalCommutator} with  
Lemma \ref{lem:GeneralDiagonalCommutator} and 
we get sixteen commutators (instead of four). (We also
do not get a norm estimate for the unitaries that make up 
the commutators.)  
\end{proof}

\begin{lem} 
 Let $\phi_k : [0,1] \rightarrow \mathbb{R}$ ($1 \leq k \leq m$)
be continuous maps such that
$$\phi_1(s) \leq \phi_2(s) \leq \phi_3(s) \leq ... \leq \phi_m(s)$$
and
$$\sum_{k=1}^m \phi_k(s) = 0$$
for all $s \in [0,1]$.

Then 
there exist $x_j, y_j \in GL^0(\M_m(C[0,1]))$ ($1 \leq j \leq 4$)
such that
$$diag(e^{\phi_1}, e^{\phi_2}, ... , e^{\phi_m})  =
\prod_{j=1}^{4} (x_j, y_j)$$
and
$$\| x_j - 1_{\A} \|, \makebox{  } \| y_j - 1_{\A} \| \leq 2 \| z - 1_{\A} \|^{1/2}$$


for $1 \leq j \leq 4$,
where
$\A =_{df} \M_m (C[0,1])$ and
$z =_{df} diag(e^{\phi_1}, e^{ \phi_2}, ..., e^{ \phi_m})$.
\label{lem:InvertibleDiagonal}
\end{lem}

\begin{proof}

   The proof is essentially the same as \cite{ThomsenCommutators} 
Lemma 2.7.  
Alternatively, the proof is the same as Lemma \ref{lem:INTDiagonal}
Part (2), but with Lemma \ref{lem:DHSDiagonalCommutator} 
replaced with \cite{ThomsenCommutators} Lemma 2.6.
\end{proof}

Recall the definitions of ``TAI" and ``$\INT$" from
end of the Introduction.

\begin{lem}
Let $\A$ be a unital separable simple
TAI-algebra and let 
$\{ c_n \}_{n=1}^{\infty}$ be a countable dense subset of the 
closed unit ball of $\A$.   

Let $\{ \I_n \}_{n=1}^{\infty}$ be a sequence of C*-subalgebras of $\A$, with
$\I_n \in \INT$ for all $n$ and let $\{ p_n \}_{n=1}^{\infty}$ be a sequence
of projections in $\A$ with $1_{\I_n} = p_n$ for all $n \geq 1$
such that for all $n \geq 1$,  the following hold:
\begin{enumerate}
\item $\tau(1- p_n) < 1/n$ for all $\tau \in T(\A)$, 
\item $\| p_n c_k - c_k p_n \| < 1/n$ for all $k \leq n$, and  
\item $p_n c_k p_n$ is within $1/n$ of an element of $\I_n$ for all $k \leq n$. 
\end{enumerate}
 Suppose that $a \in \A$ is an element such that 
$|\tau(a)| < \epsilon$ for all $\tau \in \A$.
For all $n \geq 1$, let $a_n \in \I_n$ such that 
$\| p_m a p_m - a_m \| \rightarrow 0$ as $m \rightarrow \infty$.

   Then there exists $N \geq 1$ such that for all $n \geq N$,
for all $\tau \in T(\I_n)$,
$|\tau(a_n)| < \epsilon$.  
\label{lem:AsymptoticSubalgebras}
\end{lem}

\begin{proof}

Firstly, since the map $T(\A) \rightarrow \mathbb{C} : \tau \mapsto \tau(a)$
is a continuous function on the compact set $T(\A)$, 
let $0 \leq \delta < \epsilon$ be such that 
$\delta = \max \{ |\tau(a)| : \tau \in T(\A) \}$.  
I.e., $|\tau(a)| \leq \delta < \epsilon$ for all $\tau \in T(\A)$.

Suppose, to the contrary, that $\{ n_l \}_{l=1}^{\infty}$ is a subsequence
of the positive integers and for all $l \geq 1$, $\tau_l \in T(\I_{n_l})$
is such that $|\tau_l(a_{n_l})| \geq \epsilon$.

Let $\prod_{l=1}^{\infty} \I_{n_l}$ and 
${\sum^{\oplus}_{l=1}}^{\infty} \I_{n_l}$ be the ($l_{\infty}$) direct product
and ($c_0$) direct sum respectively.   
For each $k \geq 1$, $\tau_{k}$ induces an element $\tilde{\tau}_{k} \in 
T(\prod_{l=1}^{\infty} \I_{n_l})$ in the following manner:
For $\{ b_l \}_{l=1}^{\infty} \in \prod_{l=1}^{\infty} \I_{n_l}$,
$\tilde{\tau}_{k}(\{ b_l \}_{l=1}^{\infty}) =_{df} \tau_{k}(b_{k})$. 

Since $T(\prod_{l=1}^{\infty} \I_{n_l})$ is compact, 
$\{ \tilde{\tau}_{l} \}_{l=1}^{\infty}$ must have a converging subnet
$\{ \tilde{\tau}_{l_{\alpha}} \}$.
Suppose that $\lim_{\alpha} \tilde{\tau}_{l_{\alpha}} = 
\mu \in T(\prod_{l=1}^{\infty} \I_{n_l})$.   Note that 
${\sum^{\oplus}}_{l=1}^{\infty} \I_{n_l}$ is contained in the kernel of
$\mu$.  Hence, $\mu$ naturally induces a trace in 
$T(\prod_{l=1}^{\infty} \I_{n_l}/{\sum^{\oplus}}_{l=1}^{\infty} \I_{n_l})$,
which we also denote by ``$\mu$".   

    Let $\Phi : \A \rightarrow \prod_{l=1}^{\infty} 
\I_{n_l}/{\sum^{\oplus}}_{l=1}^{\infty} \I_{n_l}$ be the unital *-embedding 
that is defined as follows:

   Let $d \in \A$ be given.
Then 
$$\Phi(d) =_{df} [\{ d_l \}_{l=1}^{\infty}]$$
where $d_l \in \I_{n_l}$ for all $l \geq 1$,
$\| p_{n_l} d p_{n_l} - d_l \| \rightarrow 0$ as $l \rightarrow \infty$,
and $[\{ d_l \}_{l=1}^{\infty}]$ is the equivalence class of
$\{ d_l \}_{l=1}^{\infty}$ in $\prod_{l=1}^{\infty} 
\I_{n_l}/{\sum^{\oplus}}_{l=1}^{\infty} \I_{n_l}$.  (It is clear that $\Phi$
is a well-defined unital *-homomorphism; in particular, $\Phi(d)$ is
independent of the choice of the
sequence $\{ d_l \}$ with the above properties.) 

Then $\mu \circ \Phi \in T(\A)$.
Then $|\mu \circ \Phi (a)| = |\mu ([ \{ a_l \}_{l=1}^{\infty} ])| 
= \lim_{\alpha} |\tilde{\tau}_{l_{\alpha}}(\{ a_l \}_{l=1}^{\infty})| = 
 \lim_{\alpha} |\tau_{l_{\alpha}}(a_{l_{\alpha}})| 
\geq \epsilon$.
This contradicts our assumption that 
$|\tau(a ) | \leq \delta < \epsilon$ for all $\tau \in T(\A)$. 
\end{proof}

The next lemma is a straightforward computation. 

\begin{lem}  
Let $n \in \mathbb{Z}_+ \cup \{ \infty \}$. 
Let $\A$ be a unital C*-algebra, 
let $V$ be a Banach space and 
let $\tau : \A \rightarrow V$ be a tracial continuous linear map. 
Let  
$\xi : [t_0, t_1] \rightarrow  GL_{n}^0 (\A)$ (or $U_{n}^0 (\A)$)
be a piecewise continuously
differentiable curve with $\xi(t_0) = 1$.

For every $\epsilon > 0$, there exists $\delta > 0$ such that the following
hold:

If $x \in GL_{n}^0(\A)$ (resp. $U_{n}^0 (\A)$)   
is such that $\|x - \xi(t_1) \| < \delta$,
then there exists a piecewise continuously differentiable curve 
$\eta : [t_0, t_1] \rightarrow GL_{n}^0 (\A)$ (resp. $U_{n}^0 (\A)$)
with 
$\eta(t_0) = 1$ and $\eta(t_1) = x$  such that
$$\| \tDS_{\tau}(\xi) - \tDS_{\tau}(\eta) \| < \epsilon$$   
\label{lem:DeterminantContinuityPart1}
\end{lem}

For a unital C*-algebra $\A$, recall that $E_u$ is the Banach space
quotient $E_u =_{df} \A/ \overline{[\A, \A]}$.
Viewing $E_u$ as a metric group (with metric induced by the norm),
$T(K_0(\A)) \subseteq E_u$ is a (not necessarily closed) topological
subgroup of $E_u$.
The metric on $E_u$ induces a pseudometric $d$ on the quotient
group $E_u/ T(K_0(\A))$.
I.e., for all $a, b \in E_u$,
$$d([a],[b]) =_{df} inf \{ \| a - b + c \| : c \in T(K_0(\A)) \},$$ 
where $\| . \|$ is the norm on $E_u$ and
$[a], [b]$ are the equivalence classes of $a, b$ (respectively)
in $E_u/T(K_0(\A))$.

\begin{lem} 
Let $\A$ be a unital C*-algebra and let $d$ be the pseudometric on 
$E_u/T(K_0(\A))$ induced by the metric (or norm) on $E_u$. 
Let $y \in GL^0_{\infty}(\A)$ (resp. $U^0_{\infty}(\A)$)  be such that
$\Delta_T(y) = 0$.

Then for every $\epsilon > 0$, there exists $\delta > 0$
such that 
if $x \in GL^0_{\infty}(\A)$ (resp. $U^0_{\infty}(\A)$) 
is such that $\| x - y \| < \delta$ then 
$$d( \Delta_T(y),0)  < \epsilon$$
\label{lem:DeterminantContinuityAt0} 
\end{lem}

\begin{proof}

Since $\Delta_T(y) = 0$, there exists a piecewise continuously
differentiable curve 
$\xi : [0,1] \rightarrow GL^0_{\infty}(\A)$  ($U^0_{\infty}(\A)$ resp.)
such that $\xi(0) = 1$, $\xi(1) = y$ and 
$\widetilde{\Delta}_T(\xi) = 0$.  
Now apply Lemma \ref{lem:DeterminantContinuityPart1}. 
\end{proof}

Next, we consider some results about the closure of the commutator
subgroup.   
For a topological group $G$,
recall that $DG$ is the commutator subgroup of $G$ and $\overline{DG}$ is its
closure.  For a unital C*-algebra $\A$, $\overline{D U(\A)}$
and $\overline{D U^0(\A)}$ will be the closures in the \emph{norm} topology. 

\begin{prop} Let $\A$ be a unital simple separable TAI-algebra.

Then $$\overline{D U(\A)} = \overline{D U^0(\A)}$$

\end{prop}
\begin{proof}
It suffices to prove the following:  Let $x, y \in U(\A)$.
Then $(x,y) \in \overline{D U^0(\A)}$.

Let $\epsilon > 0$ be given. 
Contracting $\epsilon$ if necessary, we may assume that
$\epsilon < 1/10$.

Since $\A$ is TAI, let $p \in \A$ be a projection with $\tau(1 - p) < 1/10$
for all $\tau \in \A$, let 
$\I \in \INT$ be a C*-subalgebra of $\A$ with $1_{\I} = p$,   
let $u_1  \in (1-p)\A (1-p)$ and $u_2 \in \I$ be unitaries 
(in their respective
C*-subalgebras) such that 
$\| y - (u_1 \oplus u_2) \| < \epsilon/10$, 
and
\begin{equation}
\| x y x^* y^* -  x (u_1 \oplus u_2)  x^* (u_1^* \oplus u_2^*)
 \| 
< \epsilon/10. 
\end{equation}

Since $\I \in \INT$, there exist $k, l \geq 0$ and positive
integers $n_1, n_2, ..., n_k, m_1, m_2, ..., m_l \geq 1$ such that
$\I = \bigoplus_{i=1}^k \M_{n_i}(C[0,1]) 
\oplus \bigoplus_{j=1}^l \M_{m_{j}}$.

For simplicity, let us assume that $k \geq 1$ and $l=0$.  The proofs
for the other cases are similar (and sometimes easier).

By \cite{ThomsenCircle} Lemma 1.9, for 
$1 \leq i \leq k$ and $1 \leq j \leq n_i$,
let $\alpha_{i,j} : [0,1] \rightarrow \mathbb{T}$ be a continuous map,
and let $\{ p_{i,j} \}_{1 \leq i \leq k, \makebox{  } 1 \leq j \leq n_i}$
be a collection of nonzero pairwise orthogonal projections in $\I$ such that 
$u'_2 =_{df} \sum_{1 \leq i \leq k, \makebox{  } 1 \leq j \leq n_i}
\alpha_{i,j} p_{i,j} \in \I$ and 
$$\| u'_2 - u_2 \| < \epsilon/10.$$
  
Hence,
\begin{equation}
\| xyx^*y^* -  x (u_1 \oplus u'_2) x^* (u_1^* \oplus {u'_2}^*) 
\| < 3 \epsilon/10. 
\label{Jan19.2012}
\end{equation}
 
Since $k \geq 1$, $p_{1,1}$ is a nonzero projection. 

Since $\A$ is TAI, there exists a  unitary $w \in p_{1,1} \A p_{1,1}$
such that 
$\widetilde{w} =_{df} w + (1_{\A} - p_{1,1})$ is homotopy equivalent to 
$x$ in $U(\A)$.
(Sketch of 
construction of $w$:  Since $\A$ is TAI and since $C(\mathbb{T})$ (the
universal C*-algebra generated by a unitary) is semiprojective,
we can find projection $r \in \A$ with $r \prec p_{1,1}$, a C*-subalgebra
$\B \subset \A$ with $\B \in \INT$ and $1_{\B} = 1 - r$, and unitaries
$w'' \in U_0(r \A r)$ and $w''' \in \B$ with 
$\| x - (w'' \oplus w''') \| < 1/10$.
Then $x$ is homotopy equivalent to $w'' \oplus w'''$ which in turn is
homotopy equivalent to $w'' \oplus (1-r)$.
Since $r \prec p_{1,1}$ and since $\A$ has cancellation of projections 
(\cite{LinTAI} Corollary 4.6), we can conjugate 
$w'' \oplus (1-r)$ by a unitary to get $w + (1 - p_{1,1})$. But since 
$\A$ has stable rank one (\cite{LinTAI} Theorem 4.5), $\A$ is $K_1$-injective.
Hence, $w'' \oplus (1-r)$ (and hence $x$) is homotopy equivalent to 
$w + (1 - p_{1,1})$.) 

Note that $u_1 \oplus u'_2 = \widetilde{w}^* (u_1 \oplus u'_2) \widetilde{w}$.
 
Hence, 
\begin{eqnarray*} 
& & x (u_1 \oplus u'_2)  x^* (u_1^* \oplus {u'_2}^*)\\ 
& = & x  \widetilde{w}^*   (u_1 \oplus u'_2) 
 \widetilde{w} x^* (u_1^* \oplus {u'_2}^*) \\ 
& = & ( x \widetilde{w}^*, u_1 \oplus u'_2) 
\end{eqnarray*}

But $z =_{df} x \widetilde{w}^* \in U^0(\A)$.
From this and (\ref{Jan19.2012}), we have that 
$z \in U^0(\A)$ and
\begin{equation}
\| (x,y) - (z, u_1 \oplus u'_2) \| < 3 \epsilon/10
\label{Jan19.2012.6pm}
\end{equation}

Again, since $\A$ is TAI, let $q \in \A$ be a projection with 
$\tau(1- q) < \epsilon/10$ for all $\tau \in T(\A)$,  let 
$\J \in \INT$ be a C*-subalgebra of $\A$ with $1_{\J} = q$,   
and let $v_1 \in (1-q) \A (1- q)$, $v_2 \in \J$ be unitaries (in their
respective C*-algebras) such that 
$\| z - v_1 \oplus v_2 \| < \epsilon/10$.
(note that this implies that $v_1 \oplus v_2 \in U^0(\A)$)
and 
$$\| (z, u_1 \oplus u'_2) - (v_1 \oplus v_2, u_1 \oplus u'_2)  \| < 
\epsilon/10.$$

By applying an argument similar to the one for constructing 
$z$, we can find $z' \in U^0(\A)$ such that 
$(v_1 \oplus v_2, u_1 \oplus u'_2) = (v_1 \oplus v_2, z')$.
From the above and (\ref{Jan19.2012.6pm}),
$$\| (x,y) - (v_1 \oplus v_2, z') \| < 4 \epsilon/10.$$
Hence, 
$dist((x,y), D U^0(\A)) < 4\epsilon/10$.
Since $\epsilon > 0$ was arbitrary,
we have that $(x,y) \in \overline{D U^0(\A)}$ as required.  
 \end{proof} 

\begin{lem}
Let $\A$ be a unital separable simple TAI-algebra.
For every $\epsilon >0$, there exists $\delta > 0$
such that for every self-adjoint element $a \in \A$ with  
$|\tau(a)| < \delta$ for all $\tau \in T(\A)$,
$$dist(e^{i 2 \pi a}, \overline{DU^0(\A)}) =_{df} 
\inf \{ \| e^{i2 \pi a} - u \| : u \in \overline{DU^0(\A)} \}
< \epsilon.$$
\label{lem:DistToDU}
\end{lem}

\begin{proof}
This follows from \cite{ThomsenExactSequence}
which gives a topological group isomorphism:
\[
\Phi : U^0(\A) / \overline{DU^0(\A)} \rightarrow Aff(T(\A)) / \overline{K_0(\A)}. 
\]
The map $\Phi$ is the map induced by the de la Harpe--Skandalis determinant
(with universal trace).
Note that for a self-adjoint $a \in \A$,
$\Phi ( [ e^{i 2 \pi a} ] )  = a + \overline{T(K_0(\A))}$.
\end{proof}

We will need a uniqueness result of Lin's.  Towards this, we 
fix some notation.  
For a unital C*-algebra $\A$ and for a unitary $u \in U(\A)$, 
let $\overline{u}$ denote the image of $u$ in $U(\A)/\overline{DU(\A)}$. 
For $\overline{u}, \overline{v} \in U(\A)/\overline{DU(\A)}$,  
let 
$$dist(\overline{u}, \overline{v}) =_{df} inf \{ \| x - y \| : 
x,y \in U(\A) \makebox{  and  } \overline{x} = \overline{u},
\overline{y} =  \overline{v} \}.$$  
It follows that 
$$dist(\overline{u}, \overline{v}) = inf \{ \| uv^* - x \| :
x \in \overline{DU(\A)} \} = dist(uv^*, \overline{DU(\A)}).$$
If $\A$, $\B$ are unital C*-algebras and
$\phi : \A \rightarrow \B$ is a unital *-homomorphism, then
$\phi$ brings $U(\A)$ to $U(\B)$, and brings 
$\overline{DU(\A)}$ to $\overline{DU(\B)}$.
Hence, $\phi$ induces a topological group homomorphism
$\phi^{\ddagger} : U(\A)/\overline{DU(\A)} \rightarrow
U(\B)/\overline{DU(\B)}$.  Also, $\phi$ induces a map
$[\phi] : \underline{K}(\A) \rightarrow \underline{K}(\B)$.
(Here, $\underline{K}$ is total K-theory.  See, for example,
\cite{Linbook} Definition 5.8.13.)
Finally, if $X$ is a compact metric space and $\tau \in T(C(X))$ (tracial 
state) then, by the Riesz Representation Theorem, $\tau$ induces 
a Borel probability measure $\mu_{\tau}$ on $X$.

The following is a result of Lin in \cite{LinTAIAUE}. (Also, a 
generalized version, with the space $X$ being an arbitrary 
compact metric space, can be found in \cite{LinMostRecentAUE}).

\begin{thm}  Let $X$ be a compact metric space such that 
either $X$ is a finite CW-complex with dimension no more than one or 
$X = [0,1]^n$ (n-cube) or $X = \T^n$ (n-torus).  
Let $\epsilon > 0$, let $\F \subset C(X)$ be a finite subset and
let $F : (0,1) \rightarrow (0,1)$ be a nondecreasing map.  
Then there exist $\eta > 0$, $\delta > 0$, a finite subset 
$\G \subset C(X)$, a finite subset $\mathcal{P} \subset \underline{K}(C(X))$
and a finite subset $\mathcal{U} \subset U(\M_{\infty}(C(X)))$ satisfying
the following:

Suppose that $\A$ is a unital separable simple TAI-algebra and 
$\phi, \psi : C(X) \rightarrow \A$ are two unital *-homomorphisms such 
that 
$$\mu_{\tau \circ \phi}(O_s) \geq F(s)$$
for all $s \geq \eta$, for all open balls $O_s$ in $X$ with
radius $s$ and all $\tau \in T(\A)$; 
$$|\tau \circ \phi(g) - \tau \circ \psi(g) | < \delta$$
for all $g \in \G$ and all $\tau \in T(\A)$; and  
$$[\phi] |_{\mathcal{P}} = [\psi]|_{\mathcal{P}} \makebox{  and  }
dist(\phi^{\ddagger}(\overline{z}), \psi^{\ddagger}(\overline{z})) < \delta$$
for all $z \in \mathcal{U}$.

Then there exists a unitary $u \in \A$ such that
$$\| \phi(f) - u \psi(f) u^* \| < \epsilon$$
for all $f \in \F$.  
\label{thm:LinAUE} 
\end{thm}

\begin{proof}
This follows from \cite{LinTAIAUE} Theorem 10.8. 
\end{proof}

\begin{lem}
Let $\A$ be a unital C*-algebra.
Let $x \in GL^0(\A)$ have polar decomposition 
$x = u |x|$. (So $u$ is a unitary and $|x|$ is a positive invertible.)

Suppose that $\Delta_T(x) = 0$.

Then $\Delta_T(u) = \Delta_T(|x|) = 0$.
Moreover, 
$\tau(Log(|x|)) = 0$ for all $\tau \in T(\A)$.
\label{lem:DeterminantPolarDecomposition}
\end{lem}

\begin{proof}
This follows from the (short) argument of \cite{HarpeSkandalisI}
Proposition 2 d). 
\end{proof}

  Let $\A$, $\B$ be C*-algebras and let $\phi : \A \rightarrow \B$
be a *-homomorphism.  Then for every $n \geq 1$, 
the map $\M_n(\A) \rightarrow \M_n(\B) : [ a_{i,j} ] \mapsto
[ \phi(a_{i,j}) ]$ is a *-homomorphism, which will also
denote by ``$\phi$".  

\begin{lem}  Let $\A$ be a unital separable simple TAI-algebra.
\begin{enumerate}
\item
If $u \in U^0(\A)$ is a unitary such that 
$\Delta_T(u) = 0$, then 
for every $\epsilon > 0$,  
there exist 
unitaries $x_j, y_j \in U^0(\A)$, $1 \leq j \leq 18$, such that 
$$\| u - \prod_{j=1}^{18} (x_j, y_j) \| < \epsilon.$$

(Here, $\prod_{j=1}^{18} (x_j, y_j) = (x_1, y_1)(x_2, y_2) ... 
(x_{18}, y_{18})$.)\\   
\item  If $x \in GL^0(\A)$ is an invertible such that
$\Delta_T(x) = 0$, then
for every $\epsilon > 0$, 
there exist
invertibles $x_j, y_j \in GL^0(\A)$, $1 \leq j \leq 24$, such that
$$\| x - \prod_{j=1}^{24} (x_j, y_j) \| < \epsilon.$$  
\end{enumerate}

\label{lem:18CommutatorsApprox}
\end{lem}
 
\begin{proof}

We prove Part (1).  The proof of Part (2) is similar.

Let $X \subseteq \T$ be the compact subset given by
$$X =_{df} \{ t \in \T : |t - 1 | \leq 2 \| u - 1 \| \}.$$
(Note that $1 \in X$, and $X$ is either $\T$ or homeomorphic to $[0,1]$.) 

Let $\epsilon > 0$ be given.  
Contracting $\epsilon$ if necessary, we may assume that $0 < \epsilon < 1/10$
and that $\epsilon > 0$ is small enough so that for every unitary
$v \in \A$, if $\| u - v \| < \epsilon$ then $sp(v) \subseteq X$.  

Let $\delta_1 > 0$ be such that for all self-adjoint elements
$c, c' \in \A$ if $\| c - c' \| < \delta_1$
then $\| e^{i 2 \pi c} - e^{i 2 \pi c'} \| < \epsilon/10$.  We may assume that 
$\delta_1 < \epsilon/10$.
Plug $\delta_1/10$ (for $\epsilon$) into 
Lemma \ref{lem:DeterminantContinuityAt0} to get $\delta_2 > 0$.
We may assume that $\delta_2 < \epsilon/10$.  

By \cite{LinTAIUnitaries} Theorem 3.3, there exists a self-adjoint element
$a \in \A$ such that 
\begin{equation}
\| u - e^{i2 \pi a} \| <  \delta_2. 
\label{equ:ApplyLinExponential}
\end{equation}

By our choice of $\delta_2$,
we must have that 
$$d(\Delta_T(e^{i2 \pi a}), 0)  < \delta_1/10$$ 
where $d$ is the pseudometric on $Aff(T(\A))/ K_0(\A)$
induced by the (uniform) metric on $Aff(T(\A))$.

Since $\Delta_T(e^{i2 \pi a}) = [a]$ (where $[a]$ is the equivalence
class of $a$ in $Aff(T(\A))/ K_0(\A)$),  
there exist projections $q, r \in \M_{\infty}(\A)$
such that 
\begin{equation}
|\tau(a) - \tau(q) + \tau(r)| < \delta_1/10  
\label{equ:Feb2010.10:13AM}
\end{equation}
for all $\tau \in T(\A)$.

Let $F : (0,1) \rightarrow (0,1)$ be the nondecreasing map 
given by $F(t) =_{df} t/10$ for all $t \in (0,1)$.
Let $\F \subset C(X)$ be a finite subset that contains the identity
function $h(t) =_{df} t$ ($t \in X$).

Plug $X$, $\epsilon/10$ (for $\epsilon$), $\F$ and $F$ into Theorem
\ref{thm:LinAUE} 
to get $\eta_1 > 0$, 
$\delta_3 > 0$, a finite subset $\G \subset C(X)$, a finite subset
$\mathcal{P} \subset \underline{K}(C(X))$ and
a finite subset $\mathcal{U} \subset U(\M_{\infty}(C(X))$ satisfying 
the conclusions of Theorem \ref{thm:LinAUE}.

Note that $X$ is closed under complex conjugates. Hence, 
let $S =_{df} \{ 1, t_1, \overline{t_1},  t_2, \overline{t_2},
 ..., t_N, \overline{t_N} \} \subset X$ be a finite collection of $2N + 1$ 
distinct points and $\eta_2 > 0$ 
such that for all $s \geq  \eta_1$ for all open balls
$O_s$ in $X$ with radius $s$,
\begin{equation}
card(O_s \cap S)(1 - \eta_2)/(2N + 1) > F(s),     
\label{equ:Feb10.2012.4:25PM}
\end{equation}
where $card(O_s \cap S)$ is the cardinality of $O_s \cap S$.

Let $\Phi : C(X) \rightarrow \A$ be the unital *-homomorphism
given by $\Phi(h) =_{df} e^{i2\pi a}$,
where $h \in C(X)$ is the identity map (i.e., $h(t) = t$ for all $t \in X$).
Note that by our assumption on $\epsilon$ and by 
(\ref{equ:ApplyLinExponential}), the spectrum of $e^{i 2 \pi a}$ 
is contained in
$X$; so $\Phi$ is well-defined.   
Let $N_1 \geq 1$ be an integer so
that $\mathcal{U} \subset \M_{N_1}(C(X))$.  
Let $M_1 \geq 1$ be an integer and  let 
$\F_1 \subset \M_{N_1}(\A)$ be  
a finite set of self-adjoint elements so that 
for all $v \in \mathcal{U}$, there exist self--adjoint elements 
$a_{v,1}, a_{v,2}, ..., a_{v, M_1} \in \F_1$ (repetitions
allowed) so that $\Phi(v) = e^{i2 \pi a_{v, 1}} 
e^{i 2 \pi a_{v, 2}} ... e^{i 2 \pi a_{v, M_1}}$.
(Note that $K_1(\Phi) = 0$.)    

Choose $\delta_4 > 0$ such that if $u_1, u_2, ..., u_{M_1} \in U(\M_{N_1}(\A))$
are unitaries such that 
$dist(u_j, \overline{DU(\M_{N_1}(\A))}) < \delta_4$ for $1 \leq j \leq M_1$
then $dist(u_1 u_2 ... u_{M_1}, \overline{DU(\M_{N_1}(\A))}) < \delta_3/10$.
We may assume that $\delta_4 < \delta_3/10$.

Plug $\delta_4/10$ (for $\epsilon$) and $\M_{N_1}(\A)$ (for $\A$) 
into Lemma \ref{lem:DistToDU} to get $\delta_5 > 0$.

Choose integer $N_2 \geq 1$ such that 
$1/N_2 <  \eta_2/10$.
Also choose $N_3 \geq 1$ such that 
$N_3 \geq max  \{ \| b \| : b \in \G \cup \F_1 \}$.   

Since $\A$ is TAI and by Lemma \ref{lem:AsymptoticSubalgebras} and
Remark \ref{rem:TAIMatrixSize}, let $p \in \A$ be a projection and let 
$\I \in \INT$ be a C*-subalgebra of $\A$ with $1_{\I} = p$ 
such that the following
hold:
\begin{enumerate}
\item[(a)] $\tau(1_{\A} - p) < 
min \{ \frac{\delta_3}{10(1 + N_3)},
\frac{\delta_5}{10(1 + N_3)}, \frac{\eta_2}{10(1 + N_3)} \}$ 
for all $\tau \in T(\A)$. 
\item[(b)] Each summand in $\I$ has matrix size at least $N_2(2 N + 1)$.
(Equivalently, every irreducible represention of $\I$ has 
image with the form $M_n$ with $n \geq N_2(2N + 1)$.) 
\item[(c)] There exists $a_1 \in \I$ such that 
$\| a - ((1-p)a(1-p) + a_1) \| < \delta_1/10$
and $\| e^{i 2 \pi a} - e^{i 2 \pi ((1-p)a(1-p) + a_1)} \| < \epsilon/10$ 
\item[(d)] There exist projections $q', r' \in \M_{\infty}(\I)$
such that $|\tau(a_1) - \tau(q') + \tau(r')| < \delta_1/10$ 
for all $\tau \in T(\I)$.  
\item[(e)]   Let $\psi_0 : C(X) \rightarrow (1-p) \A (1- p)$ be the 
unital *-homomorphism given by $\psi_0(h) =_{df}  
(1- p) e^{i 2 \pi (1-p)a(1-p)} (1 - p)$,
where $h \in C(X)$ is the identity map (i.e., $h(t) = t$ for all $t \in X$).
(Note that by (\ref{equ:ApplyLinExponential}), by (c) and our
assumptions on $\epsilon$, $p$ can be chosen so that
the spectrum of $e^{i 2 \pi (1-p) a (1-p)}$ is contained in $X$; so the
map $\psi_0$ is well-defined.) 

Then   for all $v \in \mathcal{U}$,
\begin{eqnarray*}
& & \| \psi_0(v) - (1-p)e^{i 2\pi (1 - p)a_{v,1}(1-p)} 
e^{i 2\pi (1 - p)a_{v,2}(1-p)}....e^{i 2\pi (1 - p)a_{v,M_1}(1-p)}(1-p) \| \\
& < &  \delta_3/10. 
\end{eqnarray*} 
(Here, we identify $1_{\A} -p$ with $(1_{\A} - p) \otimes 1_{\M_{N_1}}
\in \M_{N_1}(\A)$.) 
\end{enumerate}
We denote the above statements by ``(*)".

  Since  $\I \in \INT$, let us suppose, to simplify notation,
that $\I$ has the form
$$\I = \bigoplus_{j=1}^{N_4} \M_{m_j}(C[0,1])$$ 
where $N_4 \geq 1$.  The proof for the other cases are similar.

  We now construct two unital *-homomorphisms  
$\phi_1, \phi_2 : C(X) \rightarrow \A$.  

   By (*), we have that for $1 \le j \leq N_4$,
$m_j \geq N_2 (2N + 1)$.  
For each $j$,
let 
$\psi_j : C(X) \rightarrow \M_{m_j}(C[0,1])$ be the
(finite rank) unital *-homomorphism given by 
\begin{eqnarray*}
& & \psi_j(f) \\
& =_{df} & 
diag(f(1), f(t_1), f(\overline{t_1}), f(t_2), f(\overline{t_2}), ....,
f(t_N), f(\overline{t_N}), f(1), f(t_1),\\ 
& & f(\overline{t_1}), f(t_2), 
f(\overline{t_2}), ..., f(t_N), f(\overline{t_N}), f(1), f(t_1),
f(\overline{t_1}),.....)
\end{eqnarray*}
for all $f \in C(X)$, 
where the tail of the diagonal either has the form
``$...f(t_l), f(\overline{t_l}))$" or has the form
``$...f(t_l), f(\overline{t_l}), f(1))$".

Let $h \in C(X)$ be the identity function, i.e., $h(t) = t$ for
all $t \in X$.

We define the unital *-homomorphisms $\phi_1, \phi_2 : C(X) \rightarrow \A$ 
in the following manner:

$$\phi_1(h) =_{df} \psi_0(h)   
\oplus \bigoplus_{j=1}^{N_4} \psi_j(h)$$ 
and
$$\phi_2(h) =_{df} (1- p) \oplus  \bigoplus_{j=1}^{N_4} \psi_j(h).$$

   From (*), (\ref{equ:Feb10.2012.4:25PM})
and our choices of $N_2$, $N_3$  and $\eta_2$, 
we have the following statements:
\begin{equation} \label{equ:Feb10.2012.5:47PM} \end{equation}
\begin{enumerate}
\item[(i.)] $\mu_{\tau \circ \phi_2}(O_s) \geq F(s)$ for all
$s \geq \eta_1$, for all open balls $O_s$ in $X$ with radius $s$ and
for all $\tau \in T(\A)$. 
\item[(ii.)]  $| \tau \circ \phi_1(f) - \tau \circ \phi_2(f) | < \delta_3/2$
for all $f \in \G$ and for all $\tau \in T(\A)$.   
\end{enumerate}

   Next, since $X$ is either $\mathbb{T}$ or homeomorphic to $[0,1]$
and since the image of $h$ (under both $\phi_1$ and $\phi_2$) is 
contained in $U^0(\A)$,
\begin{equation}
\underline{K}(\phi_1) = \underline{K}(\phi_2).
\label{equ:Feb10.2012:5:48PMOOPS}
\end{equation} 

   Finally, from (*) (a), we have that 
$|\tau((1-p)b(1-p))| < \delta_5/10$ for all $b \in \F_1$ and for all
$\tau \in T(\M_{N_1}(\A))$.       
It follows, from the definition of $\delta_5$ and Lemma \ref{lem:DistToDU},
that $dist(e^{i 2 \pi (1-p) b (1- p)}, \overline{DU(\M_{N_1}(\A))}) < 
\delta_4/10$ for all $b \in \F_1$.
From the definition of $\delta_4$ and the definition of $\F_1$,
it follows that for all $v \in \mathcal{U}$,
$$dist(e^{i2 \pi (1-p)a_{v,1}(1-p)}e^{i2 \pi (1-p)a_{v,2}(1-p)} ...
e^{i2 \pi (1-p)a_{v,M_1}(1-p)}, \overline{DU(\M_{N_1}(\A))}) < \delta_3/10.$$
From this and (*) (e), we have that for all $v \in \mathcal{U}$,
\begin{equation}
dist(\psi_0(v) \oplus p, \overline{DU(\M_{N_1}(\A))}) < \delta_3/5.  
\label{equ:Feb10.2012.5:48PM}
\end{equation}
Also, $(1-p)\phi_2(\M_{N_1}(C(X)))(1 - p) \subseteq 
\M_{N_1}(\mathbb{C}(1-p))$ .
Hence, for all $v \in \mathcal{U}$, 
there exists self--adjoint $c \in \M_{N_1}(\mathbb{C}(1-p))$  
with $\| c \| \leq 1$ such that 
$(1-p)\phi_2(v)(1-p) \oplus p = e^{i 2 \pi c}$.
Note that this and (*) (a) implies that 
$|\tau(c)| < \delta_5/10$ for all $\tau \in T(\M_{N_1}(\A))$.    
From this, the definition of $\delta_5$ and since
$\delta_4 < \delta_3/10$,
we have that for all $v \in \mathcal{U}$,
$dist((1-p)\phi_2(v)(1-p) \oplus p, \overline{DU(\M_{N_1}(\A))}) < \delta_3/10$.
From this, the definitions of $\phi_1$,  $\phi_2$ and (\ref{equ:Feb10.2012.5:48PM}), we have that for all $v \in \mathcal{U}$,
\begin{equation}
dist(\phi_1^{\ddagger}(\overline{v}), \phi_2^{\ddagger}(\overline{v}))
< \delta_3.
\label{equ:Feb10.2012.5:49PM}
\end{equation} 

From (\ref{equ:Feb10.2012.5:47PM}), (\ref{equ:Feb10.2012:5:48PMOOPS}),
(\ref{equ:Feb10.2012.5:49PM}) and
from Theorem \ref{thm:LinAUE}, 
there exists a unitary $w \in \A$ such that for all    
$f \in \F$,
$$\| \phi_1(f) - w \phi_2(f) w^* \| < \epsilon/10.$$

Since the identity function $h$ (i.e., $h(t) =_{df} t$ for all $t \in X$) is
an element of $\F$,
it follows that 
$$\| ((1- p)e^{i 2\pi (1-p)a (1-p)}(1-p) \oplus \bigoplus_{j=1}^{N_4}
\psi_j(h))  - w ( (1-p) \oplus \bigoplus_{j=1}^{N_4} \psi_j(h)) w^* \|
< \epsilon/10.$$

From this and Corollary \ref{cor:MatrixCommutator}, 
there exist unitaries $x_1, y_1, x_2, y_2 \in \A$ such that 
\begin{equation}
\| (x_1, y_1) - ((1- p)e^{i 2\pi (1-p)a (1-p)}(1-p) \oplus \bigoplus_{j=1}^{N_4}
\psi_j(h))\| < \epsilon/10  
\label{equ:Feb10.2012.4:32PM}
\end{equation}
and
\begin{equation}
(x_2, y_2) = (1-p) \oplus \bigoplus_{j=1}^{N_4} \overline{\psi_j(h)} 
\label{equ:Feb10.2012.4:35PM}
\end{equation}

    By \cite{ThomsenCircle} Lemma 1.9,
there exist real-valued continuous functions
$\theta_{j,k} : [0,1] \rightarrow \mathbb{R}$ ($1 \leq j \leq N_4$,
$1 \leq k \leq m_j$), and there exist pairwise orthogonal
minimal projections $p_{j,k} \in \M_{m_j}(C[0,1])$ (again
$1 \leq j \leq N_4$, $1 \leq k \leq m_j$) with
$\sum_{k=1}^{m_j} p_{j,k} = 1_{\M_{m_j}(C[0,1])}$ for
$1 \leq j \leq N_4$ such that 
(a) $\theta_{j,1} \leq \theta_{j,2} \leq ... \leq \theta_{j,m_j}$
for $1 \leq j \leq N_4$ and
(b) $\sum_{j=1}^{N_4} \sum_{k=1}^{m_j} \theta_{j,k} p_{j,k}$ is 
approximately unitarily equivalent to 
$a_1$ in $\I$.   
Note that the spectrum of $\I$ is $\widehat{\I} = 
\bigsqcup_{j=1}^{N_4} \widehat{\M_{m_j}(C[0,1])} = 
\bigsqcup_{j=1}^{N_4} [0,1]$;
and so, for all $s \in \widehat{\I}$, the spectrum
of $a_1(s)$ is 
$\{ \theta_{j,k}(s) : 1 \leq j \leq N_4 \makebox{  and  } 
1 \leq k \leq m_j \}$ (counting multiplicity; and where 
if $s \notin \widehat{\M_{m_j}(C[0,1])}$ then
we define $\theta_{j,k}(s) = 0$, for all $j,k$). 

    Hence,
replacing $\sum_{j=1}^{N_4} \sum_{k=1}^{m_j} \theta_{j,k} p_{j,k}$
by a unitary equivalent (in $\I$) self-adjoint element if necessary, 
we may assume that
\begin{equation}
\| a_1 - \sum_{j=1}^{N_4} \sum_{k=1}^{m_j} \theta_{j,k} p_{j,k} \|
< \delta_1/10.
\label{equ:Feb17,2012.11:03AM}
\end{equation} 
(Note that a unitary equivalence is the same as simultaneously
replacing the projections
$p_{j,k}$ by unitarily equivalent projections, with the same unitary
for all the projections. In particular, the eigenvalue functions 
$\theta_{j,k}$ stay the same.) 

Moreover, by (*) (d) and our assumptions on 
$\theta_{j,k}$, $p_{j,k}$, for all $\tau \in T(\I)$, 
\begin{equation}
\left|\tau\left(\sum_{j=1}^{N_4} \sum_{k=1}^{m_j} \theta_{j,k} p_{j,k} \right)
- \tau(q') + \tau(r') \right| < \delta_1/10.
\label{equ:10:36PMFeb15,2012}
\end{equation} 

Let $g : \widehat{\I} = \bigsqcup_{j=1}^{N_4} \widehat{\M_{m_j}(C[0,1])}
\rightarrow \mathbb{R}$
be the continuous function defined as follows:

For $s \in \widehat{\M_{m_j}(C[0,1])} \cong [0,1]$,
$$g(s) =_{df}  (1/m_j)\sum_{k=1}^{m_j} \theta_{j,k}(s) - 
(1/m_j)Tr(q'(s)) + (1/m_j)Tr(r'(s)).$$
where $Tr$ is the (nonnormalized) trace on $\M_{\infty}$.  
(Note that $q', r'$ must (by definition of $\M_{\infty}(\I)$) sit
in some big matrix algebra over $\I$.) 

Hence, $g 1_{\I} \in \I$ is a self-adjoint element, and by
(\ref{equ:10:36PMFeb15,2012}), 
\begin{equation}
\| g 1_{\I} \| < \delta_1/10
\label{equ:4:10PMFeb17,2012}
\end{equation} 
and 
\begin{equation}
\tau\left(\sum_{j=1}^{N_4} \sum_{k=1}^{m_j} \theta_{j,k} p_{j,k} \right)
- \tau(q') + \tau(r') - \tau(g 1_{\I}) = 0   
\label{equ:1:23PMFeb16,2012}
\end{equation}
for all $\tau \in T(\I)$.

For $1 \leq j \leq N_4$, fix $\tau_j \in T(\M_{m_j}(C[0,1]))$.  
Let $a_2 =_{df} \sum_{j=1}^{N_4} \sum_{k=1}^{m_j} \theta_{j,k} p_{j,k}
- g 1_{\I}$
and let 
$a_2 = \sum_{j=1}^{N_4} a_{2,j}$, where 
$a_{2,j} \in \M_{m_j}(C[0,1])$ for $1 \leq j \leq N_4$.\\
Let $q' = \sum_{j=1}^{N_4} q_j$ and
$r' = \sum_{j=1}^{N_4} r_j$ where for $1 \leq j \leq N_4$,
$q_j, r_j \in \M_{\infty} \otimes \M_{m_j}(C[0,1])$ are projections.
Suppose that, for $1 \leq j \leq N_4$,
 $q_j$,$r_j$ are the sums of $L_j$ and $L'_j$ minimal projections in
$\M_{\infty} \otimes \M_{m_j}(C[0,1])$ respectively.  
To simplify notation, let us assume that $L'_j \geq L_j$
for $1 \leq j \leq N_4$.
Then for $1 \leq j \leq N_4$,
$e^{i 2 \pi (a_{2,j} + (L'_j - L_j) p_{j,m_j})} = e^{i 2 \pi a_{2,j}})$
and hence,
\begin{equation}
e^{i 2 \pi (a_2 + \sum_{j=1}^{N_4} (L'_j - L_j) p_{j,m_j})} 
= e^{i 2 \pi a_2}. 
\label{equ:Feb17,2012.11:02AM}
\end{equation}

Let $a_3 =_{df} a_2 + \sum_{j=1}^{N_4} (L'_j - L_j) p_{j,m_j}
\in \I$.
By (\ref{equ:1:23PMFeb16,2012}), we have that     
$\tau(a_3) = 0$ for all $\tau \in T(\I)$. 
Hence, by Lemma \ref{lem:INTDiagonal} Part (1) and by \cite{ThomsenCircle}
Lemma 1.9, 
there exist unitaries $x_3, y_3, x_4, y_4, ..., x_{18}, y_{18}$
in $\A$ such that       
\begin{equation}
\| e^{i 2 \pi a_3} - (x_3, y_3)(x_4, y_4) ... (x_{18}, y_{18}) \| < 
\epsilon/10.
\label{equ:Feb17,2012.4:12PM}
\end{equation}  
(Actually, for $3 \leq j \leq 18$, 
$p x_j p, p y_j p \in \I$, 
$x_j = p x_j p \oplus (1-p)$ and $y_j = p y_j p \oplus (1-p)$.)

   From the definition of $\delta_1$ and by 
(\ref{equ:Feb17,2012.11:03AM}), (\ref{equ:4:10PMFeb17,2012}),
(\ref{equ:Feb17,2012.11:02AM})
and (\ref{equ:Feb17,2012.4:12PM}), 
$$\| e^{i 2 \pi a_1} - \prod_{j=3}^{18} (x_j, y_j) \| < \epsilon/5.$$

From this, (\ref{equ:ApplyLinExponential}),
(*) statement (c),  (\ref{equ:Feb10.2012.4:32PM})
and (\ref{equ:Feb10.2012.4:35PM}), we have that 
$$\| u - \prod_{j=1}^{18} (x_j, y_j) \| < \epsilon.$$

We now prove Part (2).
Say that $x = u |x|$ is the polar decomposition of $x$.
Then by Lemma \ref{lem:DeterminantPolarDecomposition},
$\Delta_T(u) = \Delta_T(|x|) = 0$.

By Part (1), let $x_j, y_j \in U^0(\A)$ be unitaries
such that 
\begin{equation}
\| u - \prod_{j=1}^{18} (x_j, y_j) \| < \epsilon/2.
\end{equation}

Hence, to complete the proof, it suffices to prove the following
claim:

\emph{Claim:} There exist invertibles $x_j, y_j \in GL^0(\A)$, $19 \leq j
\leq 24$, such that 
$$\| |x| - \prod_{j=19}^{24} (x_j, y_j) \| < \epsilon/2.$$

\emph{Sketch of proof of the Claim:}  The proof of the Claim is very similar to
the proof of Part (1) of this lemma.  
The main differences are the following:
\begin{enumerate} 
\item[i.]  Let $s_0 =_{df} \min ( sp(|x|) \cup sp(|x|^{-1}))  > 0$ and 
$s_1 =_{df} 
\max (sp(|x|) \cup sp(|x|^{-1}))$.
Let
$X =_{df}  [s_0/2, 2 s_1]$.   
Note that $X \subset (0, \infty)$, $1 \in X$ and for all $t > 0$,
$t \in X$ if and only if $1/t \in X$.
\item[ii.]  Since $|x| \geq 0$, $|x|$ automatically has the form
$|x| = e^a$ where $a \in \A_{sa}$.
\item[iii.] In the proof of Part (1), replace 
Corollary \ref{cor:MatrixCommutator}
and Lemma \ref{lem:INTDiagonal} with
\cite{ThomsenCommutators} Lemma 2.6 and (this paper) 
Lemma \ref{lem:InvertibleDiagonal} respectively.
\item[iv.] In the proof of Part (1), replace every occurrence 
of $\Delta_T(|x|) = 0$ and every occurrence of equation
(\ref{equ:Feb2010.10:13AM}) 
with the condition $\tau(Log(|x|) = 0$ for all 
$\tau \in T(\A)$.  (See Lemma \ref{lem:DeterminantPolarDecomposition}.)
\end{enumerate}
\emph{End of sketch of proof of the Claim.}
\end{proof}

\begin{lem}
Let $\A$ be a unital separable simple TAI-algebra.

\begin{enumerate}
\item 
Suppose that  $u \in U^0(\A)$ is a unitary such that
$\Delta_T(u) = 0$.

Then for every  $\epsilon > 0$,
there exist unitaries $x_j, y_j \in U^0(\A)$, with $1 \leq j \leq 20$, 
and there exist a self-adjoint element $a \in \A$
such that 
$$u = \left( \prod_{j=1}^{20} (x_j, y_j) \right) e^{i 2 \pi a}$$
 
$$\| a \| < \epsilon \makebox{   and   } \tau(a) = 0$$
for all $\tau \in T(\A)$.  
\item Suppose that $x \in GL^0(\A)$ is an invertible such that
$\Delta_T(x) = 0$.

  Then for every $\epsilon > 0$,
there exist invertibles $x_j, y_j \in GL^0(\A)$, with $1 \leq j \leq 26$, 
and there exist an element $d \in \A$
such that
$$u = \left( \prod_{j=1}^{26} (x_j, y_j) \right) e^d$$
$$\| d \| < \epsilon \makebox{   and   } \tau(d) = 0$$
for all $\tau \in T(\A)$.
\end{enumerate}
\label{lem:20CommutatorsApprox}  
\end{lem}

\begin{proof}
We firstly prove Part (1).  The proof of Part (2) is similar.

Choose an integer $N \geq 10$ such that if
$c_1, c_2, c_3 \in \A$ are self-adjoint elements such 
that $\| c_j \| < 1/N$ for $1 \leq j \leq 3$ then
$\| e^{i 2 \pi c_1} e^{i 2 \pi c_2} e^{i 2 \pi c_3} - 1 \| < 1$
and $(1/ 2 \pi) \| Log(e^{i 2 \pi c_1} e^{i 2 \pi c_2} e^{i 2 \pi c_3}) 
\| < \epsilon$. 

Choose a $\delta > 0$, with $\delta < 1$, such that 
for any unitary $v \in \A$, if $\| v - 1 \| < \delta$
then $(1/ 2 \pi) \| Log(v) \| < 1/(2N)$.

By Lemma \ref{lem:18CommutatorsApprox} Part (1), there exist
unitaries $x_j, y_j \in U^0(\A)$ with $1 \leq j \leq 18$ 
and there exists a unitary $w \in U^0(\A)$ such that 
\begin{equation}
u =  (x_1, y_1) (x_2, y_2) ... (x_{18}, y_{18}) w 
\label{equ:6:35PMFeb17,2012}
\end{equation} 
and $\| w - 1 \| < \delta$.

By our choice of $\delta$, there exists    
a self-adjoint element $b \in \A$ with 
$\| b \| < 1/(2N)$ such that 
$w = e^{i 2 \pi b }$. 
Since $\Delta_T(w) = 0$, 
there exist projections $p_0, q_0 \in \M_{\infty}(\A)$
such that 
$$\tau(b) - \tau(p_0) + \tau(q_0) = 0$$
for all $\tau \in T(\A)$.  

Since $\A$ is simple TAI and since $\| b \| < 1/(2N)$,  
we can replace $p_0, q_0$ by projections $p, q \in \A$
with $\tau(p), \tau(q) < 1/(2N)$ and
$$\tau(b) - \tau(p) + \tau(q) = 0$$
for all $\tau \in T(\A)$.  

   Since $\tau(p), \tau(q) < 1/(2N)$ for all $\tau \in T(\A)$
and since $\A$ has strict comparison,
there exist pairwise orthogonal projections  
$p_1, p_2, ..., p_N, q_1, q_2, ..., q_N \in \A$  
such that $p_j \sim p$ and $q_j \sim q$ for $1 \leq j \leq N$. 
Hence, 
\begin{equation}
\tau(b) - \tau \left( (1/N)\sum_{j=1}^N p_j \right) +
\tau \left( (1/N)\sum_{j=1}^N q_j \right)  = 0
\label{equ:Feb17,2012.6:36PM}
\end{equation}
for all $\tau \in T(\A)$.   

We have that 
\begin{equation}
e^{i 2\pi b} =  
e^{ i 2 \pi (1/N)\sum_{j=1}^N p_j}
e^{ -i 2 \pi (1/N)\sum_{j=1}^N q_j}
\left( e^{- i 2 \pi (1/N)\sum_{j=1}^N p_j}  
e^{ i 2 \pi (1/N)\sum_{j=1}^N q_j}   
e^{i 2 \pi b } \right).   
\label{equ:Feb17,2012.6:38PM}
\end{equation}

By \cite{ThomsenCommutators} Lemma 2.1,  
there exist unitaries 
$x_{19}, y_{19}, x_{20}, y_{20} \in U^0(\A)$ such that 
\begin{equation}
e^{ i 2 \pi (1/N)\sum_{j=1}^N p_j} = (x_{19}, y_{19})
\label{equ:Feb17,2012.6:41PM}
\end{equation}
and
\begin{equation}
e^{ -i 2 \pi (1/N)\sum_{j=1}^N q_j}
= (x_{20}, y_{20}).
\label{equ:Feb17,2012.6:42PM}
\end{equation}

Also, by our choice of $N$, 
there exists a self-adjoint element $a \in \A$ such that
\begin{equation}
e^{i 2 \pi a} = 
e^{- i 2 \pi (1/N)\sum_{j=1}^N p_j}
e^{ i 2 \pi (1/N)\sum_{j=1}^N q_j} e^{i2\pi b}.
\label{equ:Feb17,2012.6:45PM}
\end{equation}
Moreover, $\| a \| < \epsilon$;  
by 
\cite{HarpeSkandalisI} Lemma 3 (b),  (\ref{equ:Feb17,2012.6:45PM})
and (\ref{equ:Feb17,2012.6:36PM}),
$$\tau(a) = \tau(b) - \tau \left( (1/N)\sum_{j=1}^N p_j \right) +
\tau \left( (1/N)\sum_{j=1}^N q_j \right)  = 0$$ 
for all $\tau \in T(\A)$.
Finally, by (\ref{equ:6:35PMFeb17,2012}),
(\ref{equ:Feb17,2012.6:38PM}),
(\ref{equ:Feb17,2012.6:41PM}), (\ref{equ:Feb17,2012.6:42PM})
and (\ref{equ:Feb17,2012.6:45PM}),
$$u = (x_1, y_1) (x_2, y_2) (x_3, y_3) ... (x_{20}, y_{20}) e^{i 2 \pi a}.$$ 

The proof of Part (2) is very similar to the proof of Part (1).
The main difference is that we replace Lemma \ref{lem:18CommutatorsApprox}
Part (1) with Lemma \ref{lem:18CommutatorsApprox} Part (2).

\end{proof}

\begin{lem}
Let $\A$ be a unital separable simple TAI-algebra.

\begin{enumerate}
\item
Suppose that $u \in U^0(\A)$ is a unitary with $\| u - 1 \| < \sqrt{2}/100$
and $\tau(Log(u)) = 0$ for all $\tau \in T(\A)$. 

   Then for every $\epsilon > 0$, 
there exist unitaries $x_j, y_j, z \in U^0(\A)$, $1 \leq j \leq 6$,
such that 
$$u = \left( \prod_{j=1}^{6} (x_j, y_j) \right) z$$

$$\| z - 1 \| < \epsilon, \makebox{   } \tau(Log(z)) = 0$$
for all $\tau \in T(\A)$, 
and 
$$\| x_j - 1 \|, \makebox{   } \| y_j - 1 \| 
 < 2\sqrt{2} \| u - 1 \|^{1/2}$$ 
for $1 \leq j \leq 6$. 
\item  Suppose that $x \in GL^0(\A)$ is an invertible with
$\| x - 1 \| < 1/1000$ and 
$\tau(Log(x)) = 0$ for all $\tau \in T(\A)$.

   Then for every $\epsilon > 0$,
there exist invertibles $x_j, y_j, z \in GL^0(\A)$, $1 \leq j \leq 12$,
such that
$$x = \left( \prod_{j=1}^{12} (x_j, y_j) \right) z$$

$$\| z - 1 \| < \epsilon, \makebox{   } \tau(Log(z)) = 0$$
for all $\tau \in T(\A)$, and 
$$\| x_j - 1 \|, \makebox{   } \| y_j - 1 \| <  24 \| x - 1 \|^{1/2}$$ 
for $1 \leq j \leq 12$. 
\end{enumerate}
\label{lem:6CommutatorsApprox}
\end{lem}

\begin{proof}

The argument of Part (1) is a variation on the argument of Lemma
\ref{lem:18CommutatorsApprox} Part (1), where we need to control the
norm distance to the unit of the operators that make up the commutators.
We go through the proof for the convenience of the reader.

Let $X \subseteq \T$ be the compact subset given by
$$X =_{df} \{ t \in \T : |t - 1 | \leq 2 \| u - 1 \| \}.$$
(Note that $1 \in X$, and $X$ is either $\T$ or homeomorphic to $[0,1]$.)

Let $\epsilon > 0$ be given.
Contracting $\epsilon$ if necessary, we may assume that 
$0 < \epsilon < min \{ 1/100, \| u -1 \| \}$
and that $\epsilon > 0$ is small enough so that for every unitary
$v \in \A$, if $\| u - v \| < \epsilon$ then $sp(v) \subseteq X$.

Since $\| u - 1 \| < \sqrt{2}/10$, $a =_{df} (1/ (i 2 \pi )) Log(u) \in 
\A_{s a}$ and  
\begin{equation}
u = e^{i2 \pi a}.
\label{equ:CloseTo1Exponential}
\end{equation}
Hence, $\tau(a) = 0$ for all $\tau \in T(\A)$.
Also, $sp(a) \subset (-\pi/2, \pi/2)$.

Let $\delta_1 > 0$ be such that for all self-adjoint elements
$c, c' \in \A$ if $\| c - c' \| < \delta_1$
then $\| e^{i 2 \pi c} - e^{i 2 \pi c'} \| < \epsilon/10$.  We may assume that
$\delta_1 < \epsilon/10$ and that for all $0 < \delta'_1 \leq \delta_1$, 
$(\delta'_1 + sp(a)) \cup (-\delta'_1 + sp(a)) \subset (-\pi/2, \pi/2)$.

Let $F : (0,1) \rightarrow (0,1)$ be the nondecreasing map
given by $F(t) =_{df} t/10$ for all $t \in (0,1)$.
Let $\F \subset C(X)$ be a finite subset that contains the identity
function $h(t) =_{df} t$ ($t \in X$).

Plug $X$, $\epsilon/10$ (for $\epsilon$), $\F$ and $F$ into Theorem
\ref{thm:LinAUE}
to get $\eta_1 > 0$,
$\delta_3 > 0$, a finite subset $\G \subset C(X)$, a finite subset
$\mathcal{P} \subset \underline{K}(C(X))$ and
a finite subset $\mathcal{U} \subset U(\M_{\infty}(C(X))$ satisfying
the conclusions of Theorem \ref{thm:LinAUE}.

Note that $X$ is closed under complex conjugates. Hence,
let $S =_{df} \{ 1, t_1, \overline{t_1},  t_2, \overline{t_2},
 ..., t_N, \overline{t_N} \} \subset X$ be a finite collection of $2N + 1$
distinct points and $\eta_2 > 0$
such that for all $s \geq  \eta_1$ for all open balls
$O_s$ in $X$ with radius $s$,
\begin{equation}
card(O_s \cap S)(1 - \eta_2)/(2N + 1) > F(s),
\label{equ:June10,2:40PM}
\end{equation}
where $card(O_s \cap S)$ is the cardinality of $O_s \cap S$.

Let $\Phi : C(X) \rightarrow \A$ be the unital *-homomorphism
given by $\Phi(h) =_{df} e^{i2\pi a} = u$,
where $h \in C(X)$ is the identity map (i.e., $h(t) = t$ for all $t \in X$).
Note that 
the spectrum of $u = e^{i 2 \pi a}$
is contained in
$X$; so $\Phi$ is well-defined.
Let $N_1 \geq 1$ be an integer so
that $\mathcal{U} \subset \M_{N_1}(C(X))$.
Let $M_1 \geq 1$ be an integer and  let
$\F_1 \subset \M_{N_1}(\A)$ be
a finite set of self-adjoint elements so that
for all $v \in \mathcal{U}$, there exist self--adjoint elements
$a_{v,1}, a_{v,2}, ..., a_{v, M_1} \in \F_1$ (repetitions
allowed) so that $\Phi(v) = e^{i2 \pi a_{v, 1}}
e^{i 2 \pi a_{v, 2}} ... e^{i 2 \pi a_{v, M_1}}$.
(Note that $K_1(\Phi) = 0$.)

Choose $\delta_4 > 0$ such that if $u_1, u_2, ..., u_{M_1} \in U(\M_{N_1}(\A))$
are unitaries such that
$dist(u_j, \overline{DU(\M_{N_1}(\A))}) < \delta_4$ for $1 \leq j \leq M_1$
then $dist(u_1 u_2 ... u_{M_1}, \overline{DU(\M_{N_1}(\A))}) < \delta_3/10$.
We may assume that $\delta_4 < \delta_3/10$.

Plug $\delta_4/10$ (for $\epsilon$) and $\M_{N_1}(\A)$ (for $\A$)
into Lemma \ref{lem:DistToDU} to get $\delta_5 > 0$.

Choose integer $N_2 \geq 1$ such that
$1/N_2 <  \eta_2/10$.
Also choose $N_3 \geq 1$ such that
$N_3 \geq max  \{ \| b \| : b \in \G \cup \F_1 \}$.

Since $\A$ is TAI and by Lemma \ref{lem:AsymptoticSubalgebras} and
Remark \ref{rem:TAIMatrixSize}, let $p \in \A$ be a projection and let
$\I \in \INT$ be a C*-subalgebra of $\A$ with $1_{\I} = p$
such that the following
hold:
\begin{enumerate}
\item[(a)] $\tau(1_{\A} - p) <
min \{ \frac{\delta_3}{10(1 + N_3)},
\frac{\delta_5}{10(1 + N_3)}, \frac{\eta_2}{10(1 + N_3)} \}$
for all $\tau \in T(\A)$.
\item[(b)] Each summand in $\I$ has matrix size at least $N_2(2 N + 1)$.
(Equivalently, every irreducible represention of $\I$ has
image with the form $M_n$ with $n \geq N_2(2N + 1)$.)
\item[(c)] There exists $a_1 \in \I$ such that
$\| a - ((1-p)a(1-p) + a_1) \| < \delta_1/10$
and $\| e^{i 2 \pi a} - e^{i 2 \pi ((1-p)a(1-p) + a_1)} \| < \epsilon/10$.
Note that for $0 < \delta'_1 \leq \delta_1$,
$sp(a_1) \cup (\delta'_1/10 + sp(a_1)) \cup (-\delta'_1/10 + sp(a_1))
\subset (-\pi/2, \pi/2)$.  
\item[(d)] $|\tau(a_1)| < \delta_1/10$
for all $\tau \in T(\I)$.
\item[(e)]   Let $\psi_0 : C(X) \rightarrow (1-p) \A (1- p)$ be the
unital *-homomorphism given by $\psi_0(h) =_{df}
(1- p) e^{i 2 \pi (1-p)a(1-p)} (1 - p)$,
where $h \in C(X)$ is the identity map (i.e., $h(t) = t$ for all $t \in X$).
(Note that by (\ref{equ:CloseTo1Exponential}), by (c) and our
assumptions on $\epsilon$, $p$ can be chosen so that
the spectrum of $e^{i 2 \pi (1-p) a (1-p)}$ is contained in $X$; so the
map $\psi_0$ is well-defined.)

Then   for all $v \in \mathcal{U}$,
\begin{eqnarray*}
& & \| \psi_0(v) - (1-p)e^{i 2\pi (1 - p)a_{v,1}(1-p)}
e^{i 2\pi (1 - p)a_{v,2}(1-p)}....e^{i 2\pi (1 - p)a_{v,M_1}(1-p)}(1-p) \| \\
& < &  \delta_3/10.
\end{eqnarray*}
(Here, we identify $1_{\A} -p$ with $(1_{\A} - p) \otimes 1_{\M_{N_1}}
\in \M_{N_1}(\A)$.)
\end{enumerate}
We denote the above statements by ``(*)".

  Since  $\I \in \INT$, let us suppose, to simplify notation,
that $\I$ has the form
$$\I = \bigoplus_{j=1}^{N_4} \M_{m_j}(C[0,1])$$
where $N_4 \geq 1$.  The proof for the other cases are similar.

  We now construct two unital *-homomorphisms
$\phi_1, \phi_2 : C(X) \rightarrow \A$.

   By (*), we have that for $1 \le j \leq N_4$,
$m_j \geq N_2 (2N + 1)$.
For each $j$,
let
$\psi_j : C(X) \rightarrow \M_{m_j}(C[0,1])$ be the
(finite rank) unital *-homomorphism given by
\begin{eqnarray*}
& & \psi_j(f) \\
& =_{df} &
diag(f(1), f(t_1), f(\overline{t_1}), f(t_2), f(\overline{t_2}), ....,
f(t_N), f(\overline{t_N}), f(1), f(t_1),\\
& & f(\overline{t_1}), f(t_2),
f(\overline{t_2}), ..., f(t_N), f(\overline{t_N}), f(1), f(t_1),
f(\overline{t_1}),.....)
\end{eqnarray*}
for all $f \in C(X)$,
where the tail of the diagonal either has the form
``$...f(t_l), f(\overline{t_l}))$" or has the form
``$...f(t_l), f(\overline{t_l}), f(1))$".

Let $h \in C(X)$ be the identity function, i.e., $h(t) = t$ for
all $t \in X$.

We define the unital *-homomorphisms $\phi_1, \phi_2 : C(X) \rightarrow \A$
in the following manner:

$$\phi_1(h) =_{df} \psi_0(h)
\oplus \bigoplus_{j=1}^{N_4} \psi_j(h)$$
and
$$\phi_2(h) =_{df} (1- p) \oplus  \bigoplus_{j=1}^{N_4} \psi_j(h).$$

   From (*), (\ref{equ:June10,2:40PM})
and our choices of $N_2$, $N_3$  and $\eta_2$,
we have the following statements:
\begin{equation} \label{equ:June10A} \end{equation}
\begin{enumerate}
\item[(i.)] $\mu_{\tau \circ \phi_2}(O_s) \geq F(s)$ for all
$s \geq \eta_1$, for all open balls $O_s$ in $X$ with radius $s$ and
for all $\tau \in T(\A)$.
\item[(ii.)]  $| \tau \circ \phi_1(f) - \tau \circ \phi_2(f) | < \delta_3/2$
for all $f \in \G$ and for all $\tau \in T(\A)$.
\end{enumerate}

   Next, since $X$ is either $\mathbb{T}$ or homeomorphic to $[0,1]$
and since the image of $h$ (under both $\phi_1$ and $\phi_2$) is
contained in $U^0(\A)$,
\begin{equation}
\underline{K}(\phi_1) = \underline{K}(\phi_2).
\label{equ:June10B}
\end{equation}

   Finally, from (*) (a), we have that
$|\tau((1-p)b(1-p))| < \delta_5/10$ for all $b \in \F_1$ and for all
$\tau \in T(\M_{N_1}(\A))$.
It follows, from the definition of $\delta_5$ and Lemma \ref{lem:DistToDU},
that $dist(e^{i 2 \pi (1-p) b (1- p)}, \overline{DU(\M_{N_1}(\A))}) <
\delta_4/10$ for all $b \in \F_1$.
From the definition of $\delta_4$ and the definition of $\F_1$,
it follows that for all $v \in \mathcal{U}$,
$$dist(e^{i2 \pi (1-p)a_{v,1}(1-p)}e^{i2 \pi (1-p)a_{v,2}(1-p)} ...
e^{i2 \pi (1-p)a_{v,M_1}(1-p)}, \overline{DU(\M_{N_1}(\A))}) < \delta_3/10.$$
From this and (*) (e), we have that for all $v \in \mathcal{U}$,
\begin{equation}
dist(\psi_0(v) \oplus p, \overline{DU(\M_{N_1}(\A))}) < \delta_3/5.
\label{equ:June10C}
\end{equation}
Also, $(1-p)\phi_2(\M_{N_1}(C(X)))(1 - p) \subseteq
\M_{N_1}(\mathbb{C}(1-p))$ .
Hence, for all $v \in \mathcal{U}$,
there exists self--adjoint $c \in \M_{N_1}(\mathbb{C}(1-p))$
with $\| c \| \leq 1$ such that
$(1-p)\phi_2(v)(1-p) \oplus p = e^{i 2 \pi c}$.
Note that this and (*) (a) implies that
$|\tau(c)| < \delta_5/10$ for all $\tau \in T(\M_{N_1}(\A))$.
From this, the definition of $\delta_5$ and since
$\delta_4 < \delta_3/10$,
we have that for all $v \in \mathcal{U}$,
$dist((1-p)\phi_2(v)(1-p) \oplus p, \overline{DU(\M_{N_1}(\A))}) < \delta_3/10$.
From this, the definitions of $\phi_1$,  $\phi_2$ and (\ref{equ:Feb10.2012.5:48PM}), we have that for all $v \in \mathcal{U}$,
\begin{equation}
dist(\phi_1^{\ddagger}(\overline{v}), \phi_2^{\ddagger}(\overline{v}))
< \delta_3.
\label{equ:June10D}
\end{equation}

From (\ref{equ:June10A}), (\ref{equ:June10B}),
(\ref{equ:June10D}) and
from Theorem \ref{thm:LinAUE},
there exists a unitary $w \in \A$ such that for all
$f \in \F$,
$$\| \phi_1(f) - w \phi_2(f) w^* \| < \epsilon/10.$$

Since the identity function $h$ (i.e., $h(t) =_{df} t$ for all $t \in X$) is
an element of $\F$,
it follows that
$$\| ((1- p)e^{i 2\pi (1-p)a (1-p)}(1-p) \oplus \bigoplus_{j=1}^{N_4}
\psi_j(h))  - w ( (1-p) \oplus \bigoplus_{j=1}^{N_4} \psi_j(h)) w^* \|
< \epsilon/10.$$

From this and Corollary \ref{cor:MatrixCommutator},
there exist unitaries $x_1, y_1, x_2, y_2 \in \A$ such that
\begin{equation}
\| (x_1, y_1) - ((1- p)e^{i 2\pi (1-p)a (1-p)}(1-p) \oplus \bigoplus_{j=1}^{N_4}
\psi_j(h))\| < \epsilon/10
\label{equ:June10E}
\end{equation}
and
\begin{equation}
(x_2, y_2) = (1-p) \oplus \bigoplus_{j=1}^{N_4} \overline{\psi_j(h)}
\label{equ:June10F}
\end{equation}
Moreover, since $|t - 1 | \leq 2 \| u - 1 \| < \sqrt{2}$ 
for all $t \in X$,
it follows, by Corollary \ref{cor:MatrixCommutator}, that 
$\| x_j - 1 \|, \| y_j - 1 \| \leq \sqrt{2} \| u - 1 \|^{1/2}$
for $j = 1, 2$. 

   Finally, by inspection (and the definition of $X$), 
we see that there exist $b_1, b_2 \in \A_{sa}$
with $(x_j, y_j) = e^{i 2 \pi b_j}$, $\| b_j \| < 1$,  
$\tau(b_j) = 0$ for all $\tau \in T(\A)$, and 
$\| e^{i 2 \pi b_j} - 1 \| \leq 2 \| u - 1 \| \leq \sqrt{2}/50$,  
for $j = 1,2$.

    By \cite{ThomsenCircle} Lemma 1.9,
there exist real-valued continuous functions
$\theta_{j,k} : [0,1] \rightarrow \mathbb{R}$ ($1 \leq j \leq N_4$,
$1 \leq k \leq m_j$), and there exist pairwise orthogonal
minimal projections $p_{j,k} \in \M_{m_j}(C[0,1])$ (again
$1 \leq j \leq N_4$, $1 \leq k \leq m_j$) with
$\sum_{k=1}^{m_j} p_{j,k} = 1_{\M_{m_j}(C[0,1])}$ for
$1 \leq j \leq N_4$ such that
(a) $\theta_{j,1} \leq \theta_{j,2} \leq ... \leq \theta_{j,m_j}$
for $1 \leq j \leq N_4$ and
(b) $\sum_{j=1}^{N_4} \sum_{k=1}^{m_j} \theta_{j,k} p_{j,k}$ is
approximately unitarily equivalent to
$a_1$ in $\I$.
Note that the spectrum of $\I$ is $\widehat{\I} =
\bigsqcup_{j=1}^{N_4} \widehat{\M_{m_j}(C[0,1])} =
\bigsqcup_{j=1}^{N_4} [0,1]$;
and so, for all $s \in \widehat{\I}$, the spectrum
of $a_1(s)$ is
$\{ \theta_{j,k}(s) : 1 \leq j \leq N_4 \makebox{  and  }
1 \leq k \leq m_j \}$ (counting multiplicity; and where
if $s \notin \widehat{\M_{m_j}(C[0,1])}$ then
we define $\theta_{j,k}(s) = 0$, for all $j,k$).

    Hence,
replacing $\sum_{j=1}^{N_4} \sum_{k=1}^{m_j} \theta_{j,k} p_{j,k}$
by a unitary equivalent (in $\I$) self-adjoint element if necessary,
we may assume that
\begin{equation}
\| a_1 - \sum_{j=1}^{N_4} \sum_{k=1}^{m_j} \theta_{j,k} p_{j,k} \|
< \delta_1/10.
\label{equ:June10G}
\end{equation}
(Note that a unitary equivalence is the same as simultaneously
replacing the projections
$p_{j,k}$ by unitarily equivalent projections, with the same unitary
for all the projections. In particular, the eigenvalue functions
$\theta_{j,k}$ stay the same.)

Moreover, by (*) (d) and our assumptions on
$\theta_{j,k}$, $p_{j,k}$, for all $\tau \in T(\I)$,
\begin{equation}
\left|\tau\left(\sum_{j=1}^{N_4} \sum_{k=1}^{m_j} \theta_{j,k} p_{j,k} \right)
 \right| < \delta_1/10.
\label{equ:June10H}
\end{equation}

Let $g : \widehat{\I} = \bigsqcup_{j=1}^{N_4} \widehat{\M_{m_j}(C[0,1])}
\rightarrow \mathbb{R}$
be the continuous function defined as follows:

For $s \in \widehat{\M_{m_j}(C[0,1])} \cong [0,1]$,
$$g(s) =_{df}  (1/m_j)\sum_{k=1}^{m_j} \theta_{j,k}(s).$$ 

Hence, $g 1_{\I} \in \I$ is a self-adjoint element, and by
(\ref{equ:June10H}),
\begin{equation}
\| g 1_{\I} \| < \delta_1/10
\label{equ:June10I}
\end{equation}
and
\begin{equation}
\tau\left(\sum_{j=1}^{N_4} \sum_{k=1}^{m_j} \theta_{j,k} p_{j,k} \right)
 - \tau(g 1_{\I}) = 0
\label{equ:June10J}
\end{equation}
for all $\tau \in T(\I)$.

Let $a_2 =_{df} \sum_{j=1}^{N_4} \sum_{k=1}^{m_j} \theta_{j,k} p_{j,k}
- g 1_{\I}$. By (*) (c) and (\ref{equ:June10I}), $sp(a_2) \subset (-\pi/2, 
\pi/2)$.   
Also, by (\ref{equ:June10J}),   we have that 
$\tau(a_2) = 0$ for all $\tau \in T(\I)$.

Hence, by Lemma \ref{lem:INTDiagonal} Part (2) (and by 
conjugating with an appropriate permutation unitary if necessary)
there exist unitaries $x_3, y_3, x_4, y_4, x_5, y_5, x_{6}, y_{6}$
in $\A$ such that
\begin{equation}
e^{i 2 \pi a_2} = (x_3, y_3)(x_4, y_4) (x_5, y_5) (x_{6}, y_{6}) 
\label{equ:June10K}
\end{equation}
and for $3 \leq j \leq 6$,  
$\| x_j - 1 \| \leq \sqrt{2} \| e^{i 2 \pi a_2} - 1 \|^{1/2}$. 
Note that by the definition of $\delta_1$ and our assumptions on 
$\epsilon$,   
$\| e^{i 2 \pi a_2} - 1 \| \leq \| e^{i 2 \pi a_2} - e^{i 2 \pi a_1} \| 
+ \| e^{i 2 \pi a_1} - 1 \| 
\leq \epsilon/10 + \| e^{i 2 \pi ((1-p) a (1- p)) + a_1)} - 1 \|
\leq (1/10)\| u - 1 \| + \| e^{i 2 \pi ((1-p) a (1- p)) + a_1)} 
- e^{i 2 \pi a} \| + \| e^{i 2 \pi a} - 1 \| 
\leq (1/10) \| u - 1 \| + \epsilon/10 + \| u - 1 \| 
\leq 2 \| u - 1 \|$. 

Hence, $\|e^{i 2 \pi a_2} - 1 \| \leq \sqrt{2}/50$ and 
$\| x_j - 1 \| \leq 2 \| u - 1 \|^{1/2}$.
Similarly, $\| y_j - 1 \| \leq 2 \| u - 1 \|^{1/2}$ for
$3 \leq j \leq 6$. 


   From the definitions of $a_2$ and $\delta_1$,  
$$\| e^{i 2 \pi a_1} - \prod_{j=3}^{6} (x_j, y_j) \| < \epsilon/5.$$

From this, (\ref{equ:CloseTo1Exponential}),
(*) statement (c),  (\ref{equ:June10E})
and (\ref{equ:June10F}), we have that
$$\| u - \prod_{j=1}^{6} (x_j, y_j) \| < \epsilon$$
and $\| x_j - 1 \|, \| y_j - 1 \| \leq 2 \sqrt{2} \| u - 1 \|^{1/2}$
for $1 \leq j \leq 6$.

Hence, $u = \left( \prod_{j=1}^6 (x_j, y_j) \right) z$ where
$z \in U^0(\A)$ and $\| z - 1 \| < \epsilon$ (which is $< 1/100$
by our hypotheses on $\epsilon$).
Hence,
$e^{i 2 \pi a} = e^{i 2 \pi b_1} e^{i 2 \pi b_2} e^{i 2 \pi a_2} z$.
But 
$(1 - \| e^{i 2 \pi b_1} - 1 \|)(1 - \| e^{i 2 \pi b_2} - 1 \|)
(1 - \| e^{i 2 \pi a_2} - 1 \|)(1 - \| z - 1 \|) 
\geq (1 - \sqrt{2}/50)^3 99/100 > 1/2$.
Hence, by \cite{HarpeSkandalisI} Lemma 3(b),
$\tau(a) = \tau(b_1) + \tau(b_2) + \tau(a_2) + (1/ (2 \pi i)) \tau(Log(z))$
for all $\tau \in T(\A)$.
Hence, $\tau(Log(z)) = 0$ for all $\tau \in T(\A)$.

Next, we sketch the proof of Part (2).

Let $x = u |x|$ be the polar decomposition of $x$.
Since $\| x - 1 \| < 1/1000$,
$\| |x| - 1 \| < 2001/1000000$ and
$\| u - 1 \| < 3001/1000000 < \sqrt{2}/100$.

Also, by Lemma \ref{lem:DeterminantPolarDecomposition}, 
$\Delta_T(|x|) = \Delta_T(u) = 0$,
and $\tau(Log(|x|)) = 0$ for all $\tau \in T(\A)$.
Hence, by \cite{HarpeSkandalisI} Lemma 3(b),
$\tau(Log(u)) = 0$ for all $\tau \in T(\A)$.

  Hence, by Part (1), 
there exist unitaries $x'_j, y'_j, v \in U^0(\A)$, $1 \leq j \leq 6$,
such that 
$$u = \left( \prod_{j=1}^6 (x'_j, y'_j) \right) v$$
$$\| v - 1 \| < \min \{ \epsilon/10, 1/100 \}$$
$$\tau(Log(v)) = 0$$
for all $\tau \in T(\A)$, and 
$$\| x'_j - 1 \|, \makebox{   } \| y'_j - 1 \| <   2 \sqrt{2} \| u - 1 \|^{1/2} 
\leq 4 \sqrt{2} \| x - 1 \|^{1/2}$$
for $1 \leq j \leq 6$.\\

Claim:  There exist invertibles $x_j, y_j, z' \in GL^0(\A)$, 
$7 \leq j \leq 12$, such that 
$$|x| = \left( \prod_{j=7}^{12} (x_j, y_j) \right) z'$$
$$\| z' - 1 \| < \min \{ \epsilon/10, 1/100 \}, \makebox{   } 
\tau(Log(z')) = 0$$
for all $\tau \in T(\A)$,
and 
$$\| x_j - 1 \|, \makebox{   } \| y_j - 1 \| < 8 \| |x| - 1 \|^{1/2}$$
for $7 \leq j \leq 12$.\\ 

\emph{Sketch of proof of Claim.}\\
The proof is similar to the proof of Part (1) (also similar
to the proof of Lemma \ref{lem:18CommutatorsApprox}).    
Here are the main differences:
\begin{enumerate}
\item[i.] Let $s_0 =_{df} min ( sp(|x|) \cup sp(|x|^{-1}) )$ and
$s_1 =_{df} max ( sp(|x|) \cup sp( |x|^{-1}))$.   
Then take $X =_{df} [s_0/2, 2 s_1]$. 
\item[ii.] In the proof of Part (1), Corollary \ref{cor:MatrixCommutator}
and Lemma \ref{lem:INTDiagonal} Part (2)
should be replaced with \cite{ThomsenCommutators} Lemma 2.6 and
(this paper) Lemma \ref{lem:InvertibleDiagonal} respectively.   
\end{enumerate}
\emph{End of sketch of proof of the Claim.}\\

Note that from the Claim, it follows that 
for $7 \leq j \leq 12$,
$$\| x_j - 1 \| \leq 24 \| x - 1 \|^{1/2}$$
and 
$$\| y_j - 1 \| \leq 24 \| x - 1 \|^{1/2}.$$

From the above, we have that 
\begin{eqnarray*}
& & x \\
& = &  u |x|\\
& = & \left( \prod_{j=1}^6 (x'_j, y'_j) \right) v \left( 
\prod_{j=7}^{12} (x_j, y_j) \right) z'\\
& = &  \left(\prod_{j=1}^6 ( x'_j,  y'_j)\right)
\left(\prod_{j=7}^{12} ( v x_j v^*,  v y_j v^*) \right) v z'\\  
\end{eqnarray*} 

Let $z =_{df} v z'$.
Then $\| z - 1 \| < \epsilon$.
Also, 
$(1 - \| v - 1 \|)(1 - \| z' - 1 \|) \geq (99/100)^2 > 1/2$.
Hence,  
by \cite{HarpeSkandalisI} Lemma 3(b), 
$\tau(Log(z)) =  \tau(Log(v)) + \tau(Log(z')) = 0$ for all $\tau \in T(\A)$.
\end{proof}

Next, towards the proof of Theorem \ref{thm:TAIFirstTh}, we 
slightly reword \cite{HarpeSkandalisII} Lemma 5.17 for the case of 
interest: 

\begin{lem}
Let $\A$ be a unital C*-algebra with cancellation of projections
and two projections $p, q \in \A$ with $p + q = 1$ 
and $u \in \A$ a partial isometry  
such that $u^* u = p$ and $u u^* \leq q$.

Say that $x \in U^0(\A )$ with $x -1 \in p \A p$ and
$\| x - 1  \| < 1$.

Then there exist $v, w \in U^0(\A)$ and $y \in U^0( \A )$ with
$y -1 \in q \A q$
such that 
$x = (v, w)y$,
$\| y - 1 \| = \| x - 1 \|$,
$\max \{ \| v - 1 \|, \| w - 1 \| \} \leq \| x - 1 \|^{1/2}$,
and $T(Log(y)) = T(Log(x))$.
\label{lem:HSLemma5.17}
\end{lem}

\begin{proof}
This follows immediately from the statement of 
\cite{HarpeSkandalisII} Lemma 5.17, taking $qyq =_{df} u x u^* + q - u u^*$
and noting that $y$ is unitarily equivalent to $x$.  (Here cancellation
is used.)
\end{proof}

\begin{rem} In the proof Theorem \ref{thm:TAIFirstTh},   
we will repeatedly use Lemma \ref{lem:HSLemma5.17} and
\cite{HarpeSkandalisII} Proposition 5.18.  We note that in the latter, if the
C*-algebra $\A$ is infinite dimensional simple TAI, and if the 
starting unitary $x$ satisfies 
$\| x - 1 \| < 1/4$ then all the unitaries and partial unitaries
in the statement 
are in the connected component of the identity.
This follows immediately from the statement of    
\cite{HarpeSkandalisII} 
Proposition 5.18, and from the assumption that $\A$ is simple TAI.
\label{rem:UnitariesConnectedTo1}
\end{rem}

\begin{thm}
Let $\A$ be a unital separable simple TAI-algebra.

\begin{enumerate}
\item
Suppose that $u \in U^0(\A)$ is a unitary such that 
$\Delta_T(u) = 0$.

Then 
there exist unitaries $x_j, y_j \in U^0(\A)$, $1 \leq j \leq 34$, 
such that
$$u = \prod_{j=1}^{34} (x_j, y_j).$$ 
\item Suppose that $x \in GL^0(\A)$ is an invertible such that
$\Delta_T(x) = 0$.

Then 
there exist invertibles $x_j, y_j \in GL^0(\A)$, $1 \leq j \leq 46$,
such that 
$$x = \prod_{j=1}^{46} (x_j, y_j).$$ 
\end{enumerate}
\label{thm:TAIFirstTh}  
\end{thm}

\begin{proof}

  The proof 
 is a modification of the arguments of \cite{HarpeSkandalisII}
(see also \cite{ThomsenCommutators}), subtituting our lemmas
in the appropriate places.  (It is also the multiplicative version
of Thierry Fack's result in \cite{FackCommutators} for additive
commutators.)  For the convenience of the reader, we provide the
proof. 

   Firstly, since $\A$ is simple unital infinite dimensional TAI,
$\A$ has the ordered $K_0$ group of a simple unital real rank zero
C*-algebra (see \cite{LinTAI} Theorem 4.8 and \cite{ElliottGong} Theorem 
4.18).  Hence, since simple infinite dimensional real rank zero C*-algebras
are weakly divisible (\cite{PereraRordam} Proposition 5),   
$\A$ is weakly divisible; i.e., for every nonzero
projection $p \in \A$, for all $n \geq 2$, there is a unital
embedding of $\M_n \oplus \M_{n+1}$ into $p \A p$. 
From this, there exist 
projections $p_n, q_n, r_n$ in $\A$ ($n \geq 1$)
such that the following hold:

\begin{enumerate}
\item $p_1 + q_1 + r_1 = 1$.
\item $p_n \preceq q_n \preceq r_n$, $n \geq 1$.
\item $r_m \perp r_n$ when $m \neq n$.
\item $r_n = p_{n+1} + q_{n+1}$, $n \geq 1$. 
\end{enumerate}
  
(See \cite{FackCommutators} Lemma 3.6, \cite{HarpeSkandalisII} Proposition
6.1 (the proof) and \cite{ThomsenCommutators} Lemma 1.7.)

By Lemma \ref{lem:20CommutatorsApprox} (1),    
there exist $20$ commutators $C_j$ ($1 \leq j \leq 20$) in $U^0(\A)$
and $a \in \A_{sa}$ with $\| e^{i 2 \pi a} - 1 \| < 1/51200$ and
$\tau(a) = 0$ for all $\tau \in T(\A)$ such that 
$u = \left( \prod_{j=1}^{20} C_j \right) e^{i 2 \pi a}$.
By \cite{HarpeSkandalisII} Lemma 5.18 (and also Remark 
\ref{rem:UnitariesConnectedTo1}) to $e^{i 2 \pi a}$, 
there exist commutators $C_{21}, C_{22}$ in $U^0(\A)$ and a unitary
$x'_0 \in U^0(\A)$ 
such that
$u = \left( \prod_{j=1}^{22} C_j \right) x'_0$, 
$x'_0 - 1 \in (q_1 + r_1) \A (q_1 + r_1)$, 
$\| x'_0 - 1 \| < 1/6400$  
and $T(Log(x'_0)) = 0$.   
By applying \cite{HarpeSkandalisII} Lemma 5.18 to $x'_0$,
there exist commutators $C_{23}, C_{24}$ in 
$U^0(\A)$ and a unitary $x_0 \in U^0(\A)$ 
such that 
$u = \left( \prod_{j=1}^{24} C_j \right) x_0$,
$x_0 - 1 \in r_1 \A r_1$,
$\| x_0 - 1 \| < 1/800$
and $T(Log(x_0)) = 0$.

   Following the argument of \cite{HarpeSkandalisII} Proposition 6.1, 
we now construct (by induction)
unitaries $x_n, y^j_n, z^j_n$ ($n \geq 1$, $1 \leq j \leq 9$)
in $U^0(\A)$ with $x_1 = x_0$ such that 
the following hold:
\begin{equation} \label{equ:UnitarySequences}  \end{equation}
\begin{enumerate}
\item[(a)] $\| x_n - 1 \| < 1/(100 n^2)$, 
$T(Log(x_n)) = 0$ and $x_n - 1 \in r_n \A r_n$. 
\item[(b)] $\| y_n^j - 1 \|, \| z_n^j - 1 \| < 2/n$ ($1 \leq j \leq 9$) 
\item[(c)] $y_n^j - 1, z_n^j - 1 \in r_n \A r_n$ ($1 \leq j \leq 8$)
\item[(d)] $y_n^{9} - 1, z_n^{9} - 1 \in (r_n + r_{n+1})\A (r_n + r_{n+1})$
\item[(e)] $x_n = \left( \prod_{j=1}^{9} (y_n^j, z_n^j) \right) x_{n+1}$. 
\end{enumerate}

Suppose that the unitaries $\{ x_m \}_{m=1}^n, \{ y_m^j \}_{m=1}^{n-1}, 
\{ z_m^j \}_{m=1}^{n-1}$ 
($1 \leq j \leq 9$) have already been constructed with $x_1 = x_0$.

Apply Lemma \ref{lem:6CommutatorsApprox} (1) to 
$x_n + r_n - 1$ to get $x'_n, y_n^j, z_n^j \in U^0(\A)$ 
($1 \leq j \leq 6$)  
such that the following hold:
$$x_n = \left( \prod_{j=1}^6 (y_n^j, z_n^j) \right) x'_n$$
$$x'_n - 1 \in r_n \A r_n$$
$$\| x'_n - 1 \| < \frac{1}{51200(n+1)^2}$$
$$T(Log(x'_n)) = 0$$
$$y_n^j - 1, z_n^j - 1 \in r_n \A r_n$$
$$\| y_n^j - 1 \|, \| z_n^j - 1 \| < 2/n$$
for $1 \leq j \leq 6$.

Apply \cite{HarpeSkandalisII} Lemma 5.18 to $p_{n+1}, q_{n+1}$ 
and $x'_n + r_n - 1 \in U^0(r_n \A r_n)$.  (Recall that since $\A$
is TAI, $\A$ is $K_1$-injective; hence, 
$x'_n \in U^0(\A)$ implies that $x'_n + r_n - 1 \in U^0(r_n \A r_n)$.)
We then get     
$x''_n, y_n^j, z_n^j \in U^0(\A)$ ($j = 7,8$)   
such that the following hold:
$$x_n = \left( \prod_{j=1}^8 (y_n^j, z_n^j) \right) x''_n$$
$$x''_n - 1 \in q_{n+1} \A q_{n+1}$$
$$\| x''_n - 1 \| < \frac{1}{6400(n+1)^2}$$
$$T(Log(x''_n)) = 0$$
$$y_n^j - 1, z_n^j - 1 \in r_n \A r_n$$
$$\|y_n^j - 1 \|, \| z_n^j - 1 \| < 2/n$$ 
for $j = 7,8$.

Now apply 
Lemma \ref{lem:HSLemma5.17} to $q_{n+1}, r_{n+1}$ and
$x''_n + q_{n+1} + r_{n+1} - 1 \in (q_{n+1} + r_{n+1})\A (q_{n+1} + r_{n+1})$
to get $x_{n+1}, y_n^9, z_n^9 \in U^0(\A)$  
such that the following hold:
$$x_n = \left( \prod_{j=1}^9 (y_n^j, z_n^j) \right) x_{n+1}$$
$$x_{n+1} - 1 \in r_{n+1} \A r_{n+1}$$
$$\| x_{n+1} - 1 \| < \frac{1}{6400(n+1)^2}$$
$$y_n^9 - 1, z_n^9 - 1 \in (q_{n+1} + r_{n+1}) \A (q_{n+1} + r_{n+1})$$
$$\| y_n^9 - 1 \|, \| z_n^9 - 1 \| < 2/n$$

This completes the inductive construction of the sequences in 
(\ref{equ:UnitarySequences}).

Observe that since $x''_n - 1 \in q_{n+1} \A q_{n+1}$ and 
$x_{n+1} - 1 \in r_{n+1} \A r_{n+1}$ (and $x''_{n} = (y_n^9, z_n^9) x_{n+1}$),
we must have that $(y_n^9, z_n^9) - 1 \in r_n \A r_n + r_{n+1} \A r_{n+1}$.

We now modify the sequences in (\ref{equ:UnitarySequences}).   

Let $y_0^9  =_{df} z_0^9 =_{df} 1$ and for $n \geq 1$, 
$$\widetilde{y}_n^j =_{df} (y_{n-1}^9, z_{n-1}^9 ) y_n^j 
(y_{n-1}^9, z_{n-1}^9 )^*$$
$$\widetilde{z}_n^j =_{df} (y_{n-1}^9, z_{n-1}^9 ) z_n^j 
(y_{n-1}^9, z_{n-1}^9 )^*$$ 
for $1 \leq j \leq 9$.

From the observation above, we have that 
$\widetilde{y}_n^j - 1, \widetilde{z}_n^j -1 \in r_n \A r_n$ for $1 \leq 
j \leq 8$ and 
$\widetilde{y}_n^9 - 1, \widetilde{z}_n^9 -1 \in (r_n + r_{n+1}) 
\A (r_n + r_{n+1})$.

As a consequence, for $1 \leq j \leq 8$, the unitaries $y_n^j, z_n^j, 
\widetilde{y}_n^j, \widetilde{z}_n^j$ commute with the unitaries
$y_m^k, z_m^k, \widetilde{y}_m^k, \widetilde{z}_m^k$ for $1 \leq k \leq 8$
and $m \neq n$, and also for $k = 9$ and $m \notin \{ n-1, n, n+1 \}$.  
One then can prove (by induction) the following two relations:

$$x_1 = \left[ \left( \prod_{k=1}^n \widetilde{y}_k^1, 
\prod_{k=1}^n \widetilde{z}_k^1 
\right) ... 
\left( \prod_{k=1}^n \widetilde{y}_k^8, \prod_{k=1}^n \widetilde{z}_k^8 
\right) \prod_{k=1}^n ( y_k^9, z_k^9) \right] 
x_{n+1}$$ 
and
$$\prod_{k=1}^{2n} ( y_k^9, z_k^9) = 
 \left( \prod_{k=1}^n \widetilde{y}_{2k-1}^9, \prod_{k=1}^n 
\widetilde{z}_{2k-1}^9 \right)   \left( \prod_{k=1}^{n} y_{2k}^9, 
\prod_{k=1}^{n} z_{2k}^9 \right).$$

For all $n \geq 1$,
let $\overline{y}_n^j =_{df} \prod_{k=1}^{2n} \widetilde{y}_k^j$ and 
$\overline{z}_n^j =_{df} \prod_{k=1}^{2n} \widetilde{z}_k^j$ 
($1 \leq j \leq 8$), 
$\overline{y}_n^9 =_{df} \prod_{k=1}^n \widetilde{y}_{2k-1}^9$,  
$\overline{z}_n^9 =_{df} \prod_{k=1}^n \widetilde{z}_{2k-1}^9$, 
$\overline{y}_n^{10} =_{df}  \prod_{k=1}^{n} y_{2k}^9$, and
$\overline{z}_n^{10} =_{df}  \prod_{k=1}^{n} z_{2k}^9$. 

Clearly, as $n \rightarrow \infty$, the sequences
$\{ \overline{y}_n^j \}, \{ \overline{z}_n^j \}$ 
converge in $\A$ to, say, $\overline{y}_{\infty}^j, 
\overline{z}_{\infty}^j$ respectively ($1 \leq j \leq 10$).
Also, $x_{n+1} \rightarrow 1$ as $n \rightarrow \infty$.
Hence, we have 
that $x_1 = \prod_{j=1}^{10} (\overline{y}_{\infty}^j, 
\overline{z}_{\infty}^j)$.
Combining this with the above, we have that 
$$u = \left( \prod_{k=1}^{24} C_k \right) \left( \prod_{j=1}^{10}
(\overline{y}_{\infty}^j, 
\overline{z}_{\infty}^j) \right).$$
I.e., $u$ is the product of $34$ commutators in $U^0(\A)$.  

The proof of Part (2) is the same as the proof of Part (1),
except that Lemma \ref{lem:20CommutatorsApprox} (1), Lemma 
\ref{lem:6CommutatorsApprox} (1),  Lemma \ref{lem:HSLemma5.17},  
and \cite{HarpeSkandalisII} Lemma 5.18 are replaced with
 Lemma \ref{lem:20CommutatorsApprox} (2), Lemma
\ref{lem:6CommutatorsApprox} (2), \cite{HarpeSkandalisII} Lemma 5.11
and \cite{HarpeSkandalisII} Lemma 5.12 respectively.
\end{proof}

We note that the above argument is an improvement on the 
(nonetheless important and interesting)
argument of \cite{ThomsenCommutators} in that there
are uniform upper bounds (namely $34$ and $46$ for the
two cases) for the number of commutators.
(The proof in \cite{ThomsenCommutators} itself does not give
any upper bound and, conceivably, the number of commutators (in the
argument) could
get arbitrarily large depending on the unitary or invertible chosen.)
The argument in \cite{HarpeSkandalisII} 
gives an upper bound ($4$) for invertibles, but no explicit
upper bound
for unitaries  -- though the proof should lead to one.

It is an open question whether the number of commutators can
be reduced.

   In the next section, we will show that for the invertible case,
the number (presently $34$) of multiplicative commutators
can be reduced to $8$.

\section{Reducing the number of commutators}  

\begin{lem}
Let $\A$ be a unital C*-algebra and $p, q \in \A$ projections
with $p + q = 1$.

Say that $x \in GL^0(\A)$ is such that 
$pxp$, $qxq$ are invertible and
$$\| qxp(pxp)^{-1} pxq \| < \frac{1}{\| (qxq)^{-1} \|}.$$

Then there exist
\[
s = \left[ \begin{array}{cc}
p & 0 \\ qsp & q 
\end{array} \right], \makebox{     }  
t = \left[ \begin{array}{cc}
p & ptq \\ 0 & q 
\end{array} \right], 
\makebox{     }
d = \left[ \begin{array}{cc} 
pdp & 0 \\ 0 & qdq 
\end{array}
\right]
\]
in $GL^0(\A)$ such that 
$x = std$.

Moreover, we have the following:
\begin{enumerate}
\item[(a)] 
If $x$ is a positive invertible, then $pdp$ and $qdq$ are positive invertibles.
\item[(b)]   
$dist(pdp, U^0(p \A p)) = dist(pxp, U^0(p \A p))$ and  
$dist( qd q, U^0(q \A q)) \leq dist(qxq, U^0(q \A q))
+ \| qxp (pxp)^{-1} p x q \|$.  
\end{enumerate}
\label{lem:LUDI}
\end{lem}

\begin{proof}
The proof is exactly the same as that of
\cite{HarpeSkandalisII} Lemma 5.8. 
\end{proof}

\begin{lem} Let $\A$ be a unital simple separable 
C*-algebra such that either
\begin{enumerate}
\item $\A$ is a TAI-algebra, or 
\item $\A$ has real rank zero, strict comparison and cancellation of
projections. 
\end{enumerate}

Let $x \in \A$ be either 
a positive invertible or $dist(x, U^0(\A)) < 1/10$.

Then for
every nonzero projection  $r \in \A$ with $r \neq 1$,
there exists a projection $p \in \A$ with
$p \sim r$ such that
$pxp$ and $(1-p)x(1-p)$ are invertible and 
$$\| (1-p)xp (pxp)^{-1} px(1-p) \| <  
\frac{1}{\| ((1-p)x(1-p))^{-1} \|}.$$

Moreover, in the case where $dist(x, U^0(\A))< 1/10$,   
for every $\epsilon > 0$, 
we can choose $p$ so that 
$$\| (1-p)xp (pxp)^{-1} px(1-p) \| \leq \frac{(dist(x, U^0(\A)) 
+ \epsilon)^2}{  
\sqrt{ 1 - (21/10)(dist(x, U^0(\A)) + \epsilon)}}$$
and  
$dist(pxp, U^0(p \A p)), dist((1-p)x(1-p), U^0((1-p)\A (1-p))) 
\leq dist(x, U^0(\A)) + \epsilon$. 
(Note that the last quantity is bounded above by $1/10$, when $\epsilon$
is small enough.) 
\label{lem:LUDII} 
\end{lem}

\begin{proof}

  Let us first assume that $\A$ is unital simple infinite-dimensional 
and TAI.
We will prove the case where $dist(x, U^0(\A)) < 1/10$. 

Let $u \in U^0(\A)$ be such that $\| x - u \| < 1/10$.
We may assume that $\epsilon < 1/10 - \| x - u \|$.  

 Firstly, multiplying $u$ (and also $x$)
by a scalar in $\mathbb{T}$ if necessary, we 
may assume that $1 \in sp(u)$.  (Note that all relevant statements and
inequalities are preserved under such a multiplication.)

Choose $\delta > 0$ such that 
if $c,d \in \A_{sa}$ with $\| c - d \| < \delta$
then $\| e^{i 2 \pi c} - e^{i 2 \pi d} \| < \epsilon/100$.
We may assume that $\delta < \epsilon/100$ and that if 
$\alpha \in \mathbb{R}$ and $|\alpha| < \delta$ then 
$|e^{i 2 \pi \alpha} - 1 | <  \epsilon/100 < 1/100$.

   By \cite{LinTAIUnitaries} Theorem 3.3, 
let $a \in \A$ be a self-adjoint element
such that $\| u -  e^{i 2 \pi a} \| < \epsilon/100$.
Since $1 \in sp(u)$, we may assume that $0 \in sp(a)$.

   Let $f \in (-\infty, \infty) \rightarrow [0,1]$ be a continuous
function such that 
$$f(\alpha)  \begin{cases}
>0  &  \makebox{  } \alpha \in (-\delta/10, \delta/10) \\
= 0 & \makebox{   } \alpha \notin (-\delta/10, \delta/10)
\end{cases}
$$
Since $0 \in sp(a)$, $f(a) \neq 0$.  Hence, since $\A$ has (SP)
(see \cite{LinTAI} Theorem 3.2), let $e \in \A$ be a nonzero
projection such that $e \in Her(f(a))$.  Moreover, since  $\A$ is 
simple TAI and $r \neq 1$, we may choose $e$ so that
$e \prec 1 - r$. (Note that $\A$ is weakly divisible (see the argument
for the existence of $\{ p_n, q_n, r_n \}$ in the proof of Theorem
\ref{thm:TAIFirstTh}) and has strict comparison (see 
\cite{LinTAI} Theorem 4.7).)    
Hence, let $r' \in Her(1 - e)$ be a projection such that
$r' \sim r$.

Also, 
$\| (1 - e) a (1 - e) - a \| =
\| -ea - ae + eae \| \leq
\| e \chi_{(-\delta/10, \delta/10)}(a) a \| 
+ \| a \chi_{(-\delta/10, \delta/10)}(a) e \|
+ \| e \chi_{(-\delta/10, \delta/10)}(a) a e \|
< 3 \delta/10$. 
Hence, 
$$\| u - e^{i 2\pi (1-e)a(1-e)} \|
\leq \| u - e^{i 2 \pi a} \| + \| e^{i 2 \pi a} - e^{i 2 \pi (1-e) a(1-e)} \|
< \epsilon/100 + \epsilon/100 = \epsilon/50.$$
 
Since $\A$ is TAI, $Her(1-r)$ is TAI, and 
there exist a projection $p' \in Her(1-e)$ and
a C*-subalgebra $\B \subset Her(1-e)$ with $\B \in \INT$ such that 
$p' \preceq e$, $1_{\B} = 1- e - p'$, unitaries $u_1 \in U^0(p' \A p')$,
$u_2 \in U^0(\B)$ and projectons $r'' \in Her(p')$, $r''' \in \B$
such that 
$$\| (1 - e) e^{i 2 \pi (1-e) a (1-e)} (1 - e) - (u_1 \oplus u_2) \| <
\epsilon/100$$
and
$$\| r' - (r'' \oplus r''') \| < \epsilon/100.$$

Note that 
$\| x - (e \oplus u_1 \oplus u_2) \| 
\leq \| x - u \| + \| u -  e^{i 2 \pi (1 - e) a (1 - e)} \| 
+ \| e^{i 2 \pi (1- e) a (1 - e)} - (e \oplus u_1 \oplus u_2) \|
< \| x - u \| + \epsilon/50 + \epsilon/100 = \| x - u \| +  3 \epsilon/100
< 1/10$. 
We denote this computation by ``(*)". 

For simplicity, let us assume that $\B \cong \M_N(C[0,1])$. 
  
Since $r'' \leq p' \preceq e$, let $p'' \leq e$ be a projection such that
$p'' \sim r''$. 
Also, by \cite{ThomsenCircle} Lemma 1.9,    
there exist pairwise orthogonal minimal projections
$p_1, p_2, ..., p_N \in \M_N(C[0,1])$
and continuous 
functions $g_1, g_2, ..., g_N : [0,1] \rightarrow \mathbb{T}$ such that 
$\| u_2 - \sum_{j=1}^N g_j p_j \| < \epsilon/100$.

Suppose that $r'''$ is the sum of $M$ minimal projections in $\B$ (where
$M \leq N$).  Then take $p =_{df} p'' + \sum_{j = 1}^M p_j$.
Clearly, $p \sim r$.

By (*), 
$\| p x p - (p'' + \sum_{j=1}^M g_j p_j) \|
\leq \| pxp -  p (e \oplus u_1 \oplus u_2)p \|
+ \| p (e \oplus u_1 \oplus u_2)p - (p'' + \sum_{j=1}^M g_j p_j) \|
< \| x - u \| + 3 \epsilon/100 + 0 <  1/10$. 
We note in particular that since $p'' + \sum_{j=1}^M g_j p_j$ is  
in $U^0(p \A p)$,  $pxp$ is invertible and 
$dist(pxp, U^0(p \A p)) < 1/10$.

By (*), 
$\| (1 - p) x (1 - p) - 
((e - p'') \oplus u_1 \oplus \sum_{j=M+1}^N g_j p_j) \|
\leq \| (1-p)x(1-p)  - (1-p) (e \oplus u_1 \oplus u_2) (1-p)\| +
\| (1-p)  (e \oplus u_1 \oplus u_2) (1 - p) - ((e - p'') \oplus u_1 \oplus
\sum_{j=M+1}^N g_j p_j)\| < \| x - u \| +  3 \epsilon/100 + 0
< 1/10$.
Note in particular that since $(e - p'') \oplus u_1 
\oplus \sum_{j=M+1}^N g_j p_j$ is in $U^0((1-p)\A(1-p))$, 
$(1 -p) x (1-p)$ is invertible and $dist((1-p) x (1-p), U^0((1-p)\A (1-p))) 
< 1/10$.

Note that since $u$, $\epsilon$ are arbitrary, the computations
in the previous two paragraphs actually show that for every $\epsilon > 0$,
we can choose $p$ so that  
$dist(pxp, U^0(p \A p)), dist((1-p) x (1-p), U^0((1-p)\A (1-p))
\leq dist(x, U^0(\A)) + \epsilon$.

To simplify notation, let $u_3 =_{df} p'' + \sum_{j=1}^M g_j p_j$.
$\| (pxp)^*pxp - 1 \| = 
\| (pxp)^* pxp - (pxp)^* u_3 \| + \| (pxp)^* u_3 - u_3^* u_3 \| 
\leq \| (pxp)^* \| \| pxp - u_3 \| + \| (pxp)^* - u_3^* \|
< (1 + 1/10) (1/10) + 1/10 = 21/100$.
Hence,  $sp((pxp)^* pxp) \subset (1 - 21/100, 1 + 21/100)$.
Hence, 
$\| ((pxp)^* pxp)^{-1} \|  < 1/(1 - 21/100) = 100/79$.
Hence, $\| (pxp)^{-1} \| \leq 10/\sqrt{79}$.  
By a similar argument,
$\| ((1-p)x(1-p))^{-1} \| \leq 10/\sqrt{79}$.

Note that (*) and 
the computation in the previous paragraph actually shows
that $\| (pxp)^* pxp - 1 \| \leq (21/10)\|pxp -p u_3p \|
\leq (21/10)(\| x - u \| + 3 \epsilon/100)$.
Since $u$ was arbitrary, for every $\epsilon >0$, we can choose $p$ so that
$sp((pxp)^* pxp) \subset (1 - (21/10) (dist(x, U^0(\A)) + \epsilon), 
1 + (21/10) 
(dist(x, U^0(\A)) + \epsilon))$.
Hence, for every $\epsilon > 0$, we can choose $p$ so that
$\| (pxp)^{-1} \| \leq 1/\sqrt{1 - (21/10)(dist(x, U^0(\A)) + \epsilon)}$.  

Next,  $\| p x (1-p) \| \leq 
\| px(1-p) - pu(1-p) \| + 
\| p u (1 -p) - p (e \oplus u_1 \oplus u_2) (1-p) \|
+ \| p (e \oplus u_1 \oplus u_2) (1 - p) \| 
< 1/10  + 3 \epsilon/100 + 0 < 1/10 + 3/1000 = 103/1000$.
Similarly, $\| (1-p) x p \| < 103/1000$.
Hence, 
$\| (1-p) x p \| \| p x (1-p) \| <  
 10609/1000000  
< 79/100 \leq 
\frac{1}{\| (pxp)^{-1} \| \| ((1-p)x(1-p))^{-1} \|}$. 
Hence,  
$$\| (1-p)xp (pxp)^{-1} p x (1 -p) \| 
<  \frac{1}{ \|((1-p)x(1-p))^{-1}\| }.$$

Since $u$, $\epsilon$ are arbitrary, the computation of the
previous paragraph also yields that for every $\epsilon >0$, we can
choose $p$ so that 
$\| p x (1-p) \|, \| (1-p) x p \| \leq dist(x, U^0(\A)) + \epsilon$.
Hence, for every $\epsilon >0$, we can choose $p$ so that  
$$\| (1-p)x p (pxp)^{-1} p x (1-p) \| \leq  
(dist(x, U^0(\A)) + \epsilon)^2 / \sqrt{1 - 
(21/10)(dist(x, U^0(\A)) + \epsilon)}.$$ 
 
  The proof for the case where $x$ is a positive invertible is similar
(and easier).

  For the case where $\A$ has real rank zero, strict comparison
and cancellation, 
one uses that if $y$ is positive invertible or
unitary in $U^0(\A)$
then $y$ can be approximated (arbitrarily close in norm) by 
positive invertibles with finite spectrum
or unitaries with finite spectrum, respectively.
(See, for example, \cite{LinRR0Unitaries}.)   
One also uses that $\A$ has strict comparison and the Riesz property. 
\end{proof}

\begin{lem}
Let $\A$ be a unital separable simple  C*-algebra
such that either
\begin{enumerate}
\item $\A$ is TAI, or
\item $\A$ has real rank zero, strict comparison and cancellation of
projections.
\end{enumerate} 
Let $x \in \A$ be a unitary in $U^0(\A)$ or a positive invertible.

Then there exist pairwise orthogonal
projections $p_1, p_2, ..., p_{93} \in \A$ with $\sum_{j=1}^{93} p_j = 1_{\A}$
and $p_j \sim p_k$ for $1 \leq j, k \leq 47$ or $48 \leq j, k \leq 93$,
and elements $s, t, d \in GL(\A)$ such that the following hold:
\begin{enumerate}
\item $s$ is lower triangular: $s = 1 + \sum_{j > k} p_j s p_k$. 
\item $t$ is upper triangular: $t = 1 + \sum_{j < k} p_j t p_k$.   
\item $d$ is diagonal: $d = \sum_j p_j d p_j$. 
\item $x = s t d$. 
\end{enumerate}

Moreover, (if $x \in U^0(\A)$) we can choose the
projections $p_j$ so that for 
$1 \leq j \leq 93$, 
$p_j d p_j \in GL^0(p_j \A p_j)$ (in $U^0(p_j \A p_j)$ respectively). 

\label{lem:LUDIII}
\end{lem}

\begin{proof}
Since $\A$ is weakly divisible (see the second paragraph in the proof
of Theorem \ref{thm:TAIFirstTh}), 
there exists nonzero projections $p, q \in \A$ such that
$47[p] + 46[q] = [1_{\A}]$.
 
The rest of the  proof is similar to
\cite{HarpeSkandalisII} Lemma 6.4,
except that we use (this paper) Lemma \ref{lem:LUDI} and 
Lemma \ref{lem:LUDII}
instead of \cite{HarpeSkandalisII} Lemma 6.3.
For the case where $x$ is a unitary, in order to make the
induction  work, we additionally
need to use the norm estimates in Lemmas \ref{lem:LUDI} and \ref{lem:LUDII}, 
which require $\epsilon$ to be sufficiently
small (at each step of the induction). 
By inspection, taking
$\epsilon = 1/10^{1000}$ (for all the steps) will suffice.
\end{proof}

\begin{lem}  Let $\A$ be  a unital simple separable C*-algebra
such that either
\begin{enumerate}
\item $\A$ is a TAI-algebra, or
\item $\A$ has real rank zero, strict comparison and cancellation of
projections.
\end{enumerate}

Let $x \in \A$ be either a unitary in $U^0(\A)$
or a positive invertible element.

Then there exist pairwise orthogonal and pairwise (Murray-von Neumann)
equivalent projections $q_1, q_2, ..., q_{46} \in \A$ 
and elements $x_1, y_1, x_2, y_2,  z \in GL(\A)$
with
$$x = (x_1, y_1)(x_2, y_2) z$$
and
$$z - 1 \in q_1 \A q_1.$$
\label{lem:Mar1,2012.6:26PM}
\end{lem}

\begin{proof}
The argument is exactly the same as \cite{HarpeSkandalisII} Lemma 6.5,
but where we use Lemma \ref{lem:LUDIII} instead of 
\cite{HarpeSkandalisII} Lemma 6.4.  
\end{proof}
 
\begin{thm}
 Let $\A$ be  a unital simple separable  
TAI-algebra.

Let $x \in GL^0(\A)$ be such that 
$\Delta_T(x) = 0$.

Then 
there exist $x_j, y_j \in GL^0(\A)$, $1 \leq j \leq 8$, 
such that 
$$x = \prod_{j=1}^{8} (x_j, y_j).$$

If, in addition, $x$ is a unitary (in $U^0(\A)$) or a positive 
invertible, then
there exist $x_j, y_j \in GL^0(\A)$ (not necessarily unitary or positive), 
$1 \leq j \leq 4$,
such that 
$$x = \prod_{j=1}^{4} (x_j, y_j).$$
\label{thm:MainTAITh}
\end{thm}

\begin{proof}

In the case where $x$ is either a unitary (in the connected component of
the identity) or a positive invertible, the proof is exactly the
same as \cite{HarpeSkandalisII} Theorem 6.6, 
except that \cite{HarpeSkandalisII} Lemma 6.5 is replaced with 
Lemma \ref{lem:Mar1,2012.6:26PM}; and also,
\cite{HarpeSkandalisII} Proposition 6.1 is replaced with
Theorem \ref{thm:TAIFirstTh}.  

Now for the general case.  If $x \in GL^0(\A)$ is arbitrary,
let $x = u |x|$ be the polar decomposition of $x$.
Then by Lemma \ref{lem:DeterminantPolarDecomposition},
$\Delta_T(u) = \Delta_T(|x|) = 0$.
Then, by the cases for unitaries and positive invertibles,
$u$ and $|x|$ are both the product of $4$ multiplicative 
commutators.  Hence, $x$ is the product of $8$
multiplicative commutators, as required.
\end{proof}

\section{The real rank zero case}

\begin{thm} Let $\A$ be a unital simple separable C*-algebra
with real rank zero, strict comparison, and cancellation of projections.
\begin{enumerate}
\item
Suppose that $u \in U^0(\A)$ is a unitary such that
$\Delta_T(u) = 0$.

Then 
there exist unitaries $x_j, y_j \in U^0(\A)$, $1 \leq j \leq 34$,
such that
$$u = \prod_{j=1}^{34} (x_j, y_j).$$
\item Suppose that $x \in GL^0(\A)$ is an invertible such that
$\Delta_T(x) = 0$.

Then 
there exist invertibles $x_j, y_j \in GL^0(\A)$, $1 \leq j \leq 46$,
such that
$$x = \prod_{j=1}^{46} (x_j, y_j).$$
\end{enumerate}
\label{thm:RR0FirstTh} 
\end{thm}

\begin{proof}
The proof of this theorem is very similar to the proof of
Theorem \ref{thm:TAIFirstTh}.

Firstly, by \cite{LinRR0Embed}, there exists a unital simple AH-algebra
$\C$ with bounded dimension growth and real rank zero 
and a unital $*$-homomorphism
$\Phi : \C \rightarrow \A$  
such that 
$\Phi$ is an isomorphism at the level of the K-theory invariant.
I.e., we have the following:
\begin{enumerate}
\item[i.] The induced map 
$$K_*(\Phi) : (K_0(\C), K_0(\C)_+, K_1(\C), [1_{\C}]) 
\rightarrow (K_0(\A), K_0(\A)_+, K_1(\A),  [1_{\A}])$$
is an isomorphism
of ordered groups with unit. 
\item[ii.] The induced map $T(\Phi) : T(\A) \rightarrow T(\C)$ is
an affine homeomorphism.
\end{enumerate}
Replacing $\C$ with
$\Phi(\C)$ if necessary, we may assume that $\C$ is a 
unital C*-subalgebra of $\A$. We denote the above statements
by ``(+)".\\

The proof (both Parts (1) and (2))  
is exactly the same as the argument leading up to  
Theorem \ref{thm:TAIFirstTh}.  In particular, one needs prove
analogues to Lemma \ref{lem:18CommutatorsApprox}, 
Lemma \ref{lem:20CommutatorsApprox} and Lemma \ref{lem:6CommutatorsApprox}
as well as the argument of Theorem \ref{thm:TAIFirstTh} itself.
   
Here are the main additional ingredients:
\begin{enumerate}
\item[(a)] 
Since $\A$ has real rank zero, if $u \in U^0(\A)$,
then $u$ can be approximated
by unitaries with finite spectrum (\cite{LinRR0Unitaries}).  
More precisely, for every $\delta_2 > 0$,
there exists a self-adjoint element 
$a \in \A$, with finite spectrum, such that 
\begin{equation*}
\| u - e^{i 2 \pi a} \| < \delta_2.   
\end{equation*}
(E.g., 
the above statement replaces the statement (\ref{equ:ApplyLinExponential})
from Lemma \ref{lem:18CommutatorsApprox}.)  
Note also that by (+) statement i. (and since $\A$ has cancellation of
projections), 
there exists a unitary $z \in \A$ such that 
$z a z^* \in \C$ and hence, $z e^{i 2 \pi a} z^* = e^{ i 2 \pi z a z^*}
\in \C$.
We then work with $e^{i 2 \pi z a z^*}$ inside $\C$, which is
TAI.  
\item[(b)]  If $x \in GL^0(\A)$ is a \emph{positive} invertible,
then, since $\A$ has real rank zero, 
$x$ can be approximated arbitrarily close by positive invertibles
with finite spectrum.  Once more, by (+) statement i. (and since
$\A$ has cancellation), 
these positive invertibles
are unitarily equivalent to positive invertibles in $\C$, and
we work in $\C$, which is TAI. 
\end{enumerate}
\end{proof}

The reduction of commutators argument goes through with essentially
no change. 

\begin{thm} Let $\A$ be a unital separable simple C*-algebra
with real rank zero, strict comparison and cancellation of 
projections.

Let $x \in GL^0(\A)$ such that
$\Delta_T(x) = 0$.

Then 
there exist $x_j, y_j \in GL^0(\A)$, $1 \leq j \leq 8$,
such that
$$x = \prod_{j=1}^{8} (x_j, y_j).$$

If, in addition, $x$ is a unitary (in $U^0(\A)$) or a positive
invertible, then
there exist $x_j, y_j \in GL^0(\A)$ (not necessarily unitary or positive), 
$1 \leq j \leq 4$, 
such that
$$x = \prod_{j=1}^{4} (x_j, y_j).$$
\label{thm:MainRR0Th}
\end{thm}

\begin{proof}
  The proof is exactly the same as the
proof of Theorem \ref{thm:MainTAITh}, except
that Theorem \ref{thm:TAIFirstTh} is replaced 
with Theorem \ref{thm:RR0FirstTh}. Note that  
all the preliminary lemmas leading up to 
Theorem \ref{thm:MainTAITh} include the real rank zero case.
\end{proof}

 \end{document}